\documentclass[a4paper, reqno, 11pt]{amsart}
\usepackage[main=english, german]{babel}
\usepackage[foot]{amsaddr}
\usepackage{csquotes}

\usepackage{amsthm,amsfonts,amssymb,amsmath,amsxtra}

\usepackage{footmisc}
\usepackage{amscd}

\usepackage{listings}
\usepackage{bbm}

\usepackage{textcomp}

\usepackage{amscd}

\usepackage{stmaryrd}
\usepackage[new]{old-arrows}
\usepackage[all]{xy}
\SelectTips{cm}{}
\usepackage{verbatim}

\usepackage[margin=1.25in]{geometry}
\usepackage[backend=biber,maxbibnames=50]{biblatex}

\usepackage{mathrsfs}
\usepackage{hyperref}
\RequirePackage{xspace}
\RequirePackage{etoolbox}
\RequirePackage{enumitem}
\RequirePackage{tensor}
\RequirePackage{mathtools}

\newcommand{\sB}{\ensuremath{\mathscr{B}}\xspace}
\newcommand{\sC}{\ensuremath{\mathscr{C}}\xspace}

\newcommand{\sF}{\ensuremath{\mathscr{F}}\xspace}
\newcommand{\sG}{\ensuremath{\mathscr{G}}\xspace}

\newcommand{\sI}{\ensuremath{\mathscr{I}}\xspace}

\newcommand{\sL}{\ensuremath{\mathscr{L}}\xspace}

\newcommand{\sT}{\ensuremath{\mathscr{T}}\xspace}

\newcommand{\fkb}{\ensuremath{\mathfrak{b}}\xspace}

\newcommand{\fkd}{\ensuremath{\mathfrak{d}}\xspace}

\newcommand{\fkg}{\ensuremath{\mathfrak{g}}\xspace}
\newcommand{\fkh}{\ensuremath{\mathfrak{h}}\xspace}

\newcommand{\fkl}{\ensuremath{\mathfrak{l}}\xspace}

\newcommand{\fkp}{\ensuremath{\mathfrak{p}}\xspace}
\newcommand{\fkq}{\ensuremath{\mathfrak{q}}\xspace}

\newcommand{\fkt}{\ensuremath{\mathfrak{t}}\xspace}
\newcommand{\fku}{\ensuremath{\mathfrak{u}}\xspace}

\newcommand{\BA}{\ensuremath{\mathbb {A}}\xspace}

\newcommand{\BC}{\ensuremath{\mathbb {C}}\xspace}
\newcommand{\BD}{\ensuremath{\mathbb {D}}\xspace}

\newcommand{\BG}{\ensuremath{\mathbb {G}}\xspace}

\newcommand{\BN}{\ensuremath{\mathbb {N}}\xspace}

\newcommand{\BP}{\ensuremath{\mathbb {P}}\xspace}
\newcommand{\BQ}{\ensuremath{\mathbb {Q}}\xspace}
\newcommand{\BR}{\ensuremath{\mathbb {R}}\xspace}

\newcommand{\BZ}{\ensuremath{\mathbb {Z}}\xspace}

\newcommand{\CE}{\ensuremath{\mathcal {E}}\xspace}
\newcommand{\CF}{\ensuremath{\mathcal {F}}\xspace}
\newcommand{\CG}{\ensuremath{\mathcal {G}}\xspace}

\newcommand{\CI}{\ensuremath{\mathcal {I}}\xspace}

\newcommand{\CL}{\ensuremath{\mathcal {L}}\xspace}

\newcommand{\CO}{\ensuremath{\mathcal {O}}\xspace}
\newcommand{\CP}{\ensuremath{\mathcal {P}}\xspace}

\newcommand{\RI}{\ensuremath{\mathrm {I}}\xspace}

\newcommand{\RP}{\ensuremath{\mathrm {P}}\xspace}

\newcommand{\ad}{{\mathrm{ad}}}

\DeclareMathOperator{\Aut}{Aut}

\newcommand{\ch}{{\mathrm{ch}}}

\newcommand{\der}{{\mathrm{der}}}

\DeclareMathOperator{\Ext}{Ext}

\DeclareMathOperator{\Gal}{Gal}
\newcommand{\GL}{{\mathrm{GL}}}

\DeclareMathOperator{\Hom}{Hom}

\newcommand{\id}{\ensuremath{\mathrm{id}}\xspace}
\let\Im\relax

\DeclareMathOperator{\Im}{Im}
\newcommand{\Ind}{{\mathrm{Ind}}}
\newcommand{\Int}{\ensuremath{\mathrm{Int}}\xspace}

\newcommand{\inva}{\mathrm{IIP}}

\DeclareMathOperator{\Ker}{Ker}
  
\DeclareMathOperator{\Lie}{Lie}

\DeclareMathOperator{\ord}{ord}

\DeclareMathOperator{\Rep}{Rep}

\DeclareMathOperator{\Sh}{Sh}

\newcommand{\SL}{{\mathrm{SL}}}

\DeclareMathOperator{\supp}{supp}

\DeclareMathOperator{\rig}{rig}
\DeclareMathOperator{\alg}{alg}

\newcommand{\Tor}{{\mathrm{Tor}}}
\DeclareMathOperator{\Tot}{Tot}

\DeclareMathOperator{\wa}{wa}

\newcommand{\bB}{\ensuremath{\mathbf {B}}\xspace}

\newcommand{\bG}{\ensuremath{\mathbf {G}}\xspace}
\newcommand{\bH}{\ensuremath{\mathbf {H}}\xspace}

\newcommand{\bJ}{\ensuremath{\mathbf {J}}\xspace}

\newcommand{\bL}{\ensuremath{\mathbf {L}}\xspace}

\newcommand{\bP}{\ensuremath{\mathbf {P}}\xspace}

\newcommand{\bT}{\ensuremath{\mathbf {T}}\xspace}
\newcommand{\bU}{\ensuremath{\mathbf {U}}\xspace}

\newcommand{\lb}{\ensuremath{\lvert }\xspace}
\newcommand{\rb}{\ensuremath{\rvert }\xspace}

\newcommand{\ov}{\overline}

\newcommand{\br}{\breve}

\newtheorem{theorem}{Theorem}[section]
\newtheorem{proposition}[theorem]{Proposition}
\newtheorem{lemma}[theorem]{Lemma}

\newtheorem{corollary}[theorem]{Corollary}

\theoremstyle{definition}
\newtheorem{definition}[theorem]{Definition}
\newtheorem{example}[theorem]{Example}

\newtheorem{remark}[theorem]{Remark}

\newtheorem{hac}[theorem]{Assumption/Conjecture}

\numberwithin{equation}{section}

\numberwithin{equation}{section}

\setitemize[0]{leftmargin=*,itemsep=\the\smallskipamount}
\setenumerate[0]{leftmargin=*,itemsep=\the\smallskipamount}

\renewcommand{\to}{%
   \ifbool{@display}{\longrightarrow}{\rightarrow}%
   }
\let\shortmapsto\mapsto
\renewcommand{\mapsto}{%
   \ifbool{@display}{\longmapsto}{\shortmapsto}%
   }
\newlength{\olen}
\newlength{\ulen}
\newlength{\xlen}
\newcommand{\xra}[2][]{%
   \ifbool{@display}%
      {\settowidth{\olen}{$\overset{#2}{\longrightarrow}$}%
       \settowidth{\ulen}{$\underset{#1}{\longrightarrow}$}%
       \settowidth{\xlen}{$\xrightarrow[#1]{#2}$}%
       \ifdimgreater{\olen}{\xlen}%
          {\underset{#1}{\overset{#2}{\longrightarrow}}}%
          {\ifdimgreater{\ulen}{\xlen}%
             {\underset{#1}{\overset{#2}{\longrightarrow}}}
             {\xrightarrow[#1]{#2}}}}%
      {\xrightarrow[#1]{#2}}
   }
\makeatother
\newcommand{\xyra}[2][]{%
   \settowidth{\xlen}{$\xrightarrow[#1]{#2}$}%
   \ifbool{@display}%
      {\settowidth{\olen}{$\overset{#2}{\longrightarrow}$}%
       \settowidth{\ulen}{$\underset{#1}{\longrightarrow}$}%
       \ifdimgreater{\olen}{\xlen}%
          {\mathrel{\xymatrix@M=.12ex@C=3.2ex{\ar[r]^-{#2}_-{#1} &}}}%
          {\ifdimgreater{\ulen}{\xlen}%
             {\mathrel{\xymatrix@M=.12ex@C=3.2ex{\ar[r]^-{#2}_-{#1} &}}}
             {\mathrel{\xymatrix@M=.12ex@C=\the\xlen{\ar[r]^-{#2}_-{#1} &}}}}}%
      {\mathrel{\xymatrix@M=.12ex@C=\the\xlen{\ar[r]^-{#2}_-{#1} &}}}%
   }
\makeatletter
\newcommand{\xla}[2][]{%
   \ifbool{@display}%
      {\settowidth{\olen}{$\overset{#2}{\longleftarrow}$}%
       \settowidth{\ulen}{$\underset{#1}{\longleftarrow}$}%
       \settowidth{\xlen}{$\xleftarrow[#1]{#2}$}%
       \ifdimgreater{\olen}{\xlen}%
          {\underset{#1}{\overset{#2}{\longleftarrow}}}%
          {\ifdimgreater{\ulen}{\xlen}%
             {\underset{#1}{\overset{#2}{\longleftarrow}}}
             {\xleftarrow[#1]{#2}}}}%
      {\xleftarrow[#1]{#2}}
   }
\newcommand{\isoarrow}{%
   \ifbool{@display}{\overset{\sim}{\longrightarrow}}{\xrightarrow\sim}%
   }
   
\newsavebox{\lineone}
\newsavebox{\linetwo}
\newsavebox{\linethree}
\newlength{\lineonelen}
\newlength{\linetwolen}
\newlength{\linethreelen}
\newlength{\biggerlen}
\newcommand{\twolinestight}[2]{%
   \sbox{\lineone}{#1}%
   \sbox{\linetwo}{#2}%
   \settowidth{\lineonelen}{\usebox{\lineone}}%
   \settowidth{\linetwolen}{\usebox{\linetwo}}%
   \setlength{\biggerlen}{\maxof{\lineonelen}{\linetwolen}}%
   \begin{minipage}{\the\biggerlen}%
      \centering
      \usebox\lineone\\%
      \usebox\linetwo%
   \end{minipage}%
   }

\begin{document}


\title[On the sheaf cohomology of some $p$-adic period domains]{On the sheaf cohomology of some $p$-adic period domains with coefficients in certain line bundles}
\author{C. Spenke }
\thanks{The research was done while the author was employed at the Bergische Universität Wuppertal. }
\date{\today}
\maketitle	
\setcounter{tocdepth}{1}

	\begin{abstract}
		Let $p$ be a prime.
		This papers aims at investigating sheaf cohomology of a broader class of $p$-adic period domains, other then the Drinfeld's upper half space (cf. \cite{O2}). 
		Concretely, we let $\bG$ be a split connected reductive group over $\BQ_p$ and restrict our attention to the $p$-adic period domain $\sF^{\wa}$ which parametrizes the weakly admissible filtrations on the trivial $\bG$-isocrystal inside a complete flag variety $\sF$. Then, we consider sheaf cohomology of $\sF^{\wa}$ with coefficients in vector bundles which are induced by restriction of a homogeneous line bundle on $\sF$ associated to a dominant weight of $\bG$.

	 \end{abstract}

\tableofcontents
\section{Introduction}\label{s:Introduction} 

The origin of period domains lies in the work of Griffiths \cite{Gr1, Gr2}. He introduced them as certain open subspaces of generalized flag varieties over $\BC$ which parametrize polarized $\BR$-Hodge structures of a given type. Rapoport and Zink introduced period domains over $p$-adic fields \cite{RZ} in which we are interested in this paper. For this, let $K=\BQ_p$. Given a reductive group $\bG$ over $K$, a period domain over $K$ parametrizes weakly admissible filtrations on a $\bG$-isocrystal of a fixed Hodge type. It is an open admissible rigid-analytic subset of a generalized flag variety $\sF$ (cf. section \ref{s:perdom}). The prototype for a $p$-adic period domain over $K$ is Drinfeld's upper half space $\Omega^{(n+1)}$, which is the complement of all $K$-rational hyperplanes in the projective space $\BP^{n}_K$, i.e. 
$$ \Omega^{(n+1)}= \BP^{n}_K \backslash \bigcup_{ H \subsetneq K^{n+1}} \BP(H).$$ 
It arises from the trivial $\mathbf{GL}_{n+1}$-isocrystal inside the projective space $\sF=\BP^{n}_K$.  \\

Given an appropriate cohomology theory, it is a natural problem to determine cohomology groups for period domains. The starting point is the work of Drinfeld \cite{D}, who computed the first étale cohomology group of $\Omega^{(2)}$. Schneider and Stuhler \cite{SS} computed cohomology groups of $\Omega^{(n+1)}$ in the $p$-adic case for ``good'' cohomology theories. This includes the étale cohomology with torsion coefficients, not including $p$-torsion, and the de Rham cohomology. 
So far, the only results for coherent sheaf cohomology are known for Drinfeld's upper half space over $p$-adic fields. After the work of Schneider and Stuhler, it was Schneider together with Teitelbaum, who made the beginning and considered at first coefficients in the canonical bundle \cite{ST1}. Shortly afterwards, Pohlkamp \cite{Po} computed the sheaf cohomology with respect to the structure sheaf. Finally, Orlik was able to generalize these results to arbitrary $\mathbf{GL}_{n+1}$-equivariant vector bundles on Drinfeld's upper half space over $p$-adic fields which are induced by restriction of a homogeneous vector bundle $\CE$ on $\BP^{n}_K$ \cite{O2}. It turns out that by Schneider and Teitelbaum \cite{ST1}, the space of global sections $\CE(\Omega^{(n+1)})$ is a reflexive $K$-Fréchet space with a continuous $\mathbf{GL}_{n+1}(K)$-action and its strong dual is a locally analytic $\mathbf{GL}_{n+1}(K)$-representation.\\

The goal of this paper is to investigate sheaf cohomology of period domains over $p$-adic fields, other than $\Omega^{(n+1)}$, with coefficients in certain line bundles. For this, let $\bG$ be a split connected reductive group over $K$ with split maximal torus $\bT$. Further, let $\bB \supset \bT$ a Borel subgroup of $\bG$ associated to a cocharacter $\mu \in X_*(\bG)$ defined over $K$. Then, we consider the associated period domain $\sF^{\wa}$ which parametrizes weakly admissible filtrations on the trivial $\bG$-isocrystal of Hodge type $\mu$ inside the complete flag variety $\sF:=\bG/\bB$. We study the sheaf cohomology of these spaces with respect to the restriction of a homogeneous line bundle $\CE_\lambda:=\CL_\lambda \otimes \omega_{\sF}$ on $\sF$. Here, $\omega_{\sF}$ denotes the canonical bundle on $\sF$ and $\CL_\lambda$ a line bundle associated to a dominant weight $\lambda  \in X^*(\bT)$ wih respect to $\bB$ (cf. section \ref{s:Setup}). 

Let $G$ and $B$ be the $K$-valued points of $\bG$ and $\bB$, respectively. Furthermore, let $W$ be the Weyl group of $\bG$. For $w \in W$, we denote by $V_B^G(w)$ the twisted generalized locally analytic Steinberg representation of weight $w\cdot \lambda$ (cf. Definition \ref{twistedSteinberg}). It is a quotient of $\Ind^G_{B}({K_{w\cdot \lambda}}')$, the locally analytic induced representation of the dual of the simple algebraic $\bT$-representation of weight $w \cdot \lambda$.  Under the assumption of a hypothesis concerning the density of some local cohomology groups (cf. Assumption \ref{hypo1}), we prove the following result.

\begin{theorem}[Theorem \ref{theorem1}]\label{theorem1intro}
   Let $i_0:=\dim\sF-\lb\Delta\rb$. The homology of the (chain) complex 
   \begin{equation}\label{Cbullet}
      C_\bullet:  \bigoplus_{\substack{w \in \Omega_\emptyset \\ l(w)=\dim Y_\emptyset}}V^G_B(w) \leftarrow \ldots \leftarrow \bigoplus_{\substack{w \in \Omega_\emptyset \\ \ l(w)=1}} V_B^G(w) \leftarrow  V^G_B(\lambda) 
   \end{equation}
   starting in degree $i_0$ coincides with $H^*(\sF^{\wa},\CE_\lambda)'$, i.e. $H_*(C_\bullet)=H^*(\sF^{\wa},\CE_\lambda)'$.
\end{theorem}

Here, $\Delta$ is the set of simple roots of the root system of $\bG$ with respect to $\bB$. Further, $\Omega_\emptyset$ is a subset of $W$  defined by some numerical conditions (cf. (\ref{omegaI})) and $Y_\emptyset$ a union of Schubert cells in $\sF$ indexed by $\Omega_\emptyset$ which is closed in $\sF$ (cf. subsection \ref{s:GeoProp}). Moreover, $H^i(\sF^{\wa},\CE_\lambda)'$ denotes the strong dual of $H^i(\sF^{\wa},\CE_\lambda)$ in the sense of Schneider and Tei\-telbaum (cf. \cite[p. 50]{S}). \\

We briefly explain the idea of the proof. For this let be $Y:=\sF^{\ad} \backslash \sF^{\wa}$ with closed embedding $\iota:Y\hookrightarrow \sF^{\mathrm{ad}}$.  The main ingredient is the acyclic complex 
\begin{equation}\label{fundamentalcomplexintro}
   0 \longrightarrow \BZ \longrightarrow \bigoplus\limits_{\substack{I \subset \Delta \\ \lb \Delta \backslash I \rb = 1 }} \BZ_I \longrightarrow \bigoplus\limits_{\substack{I \subset \Delta \\ \lb \Delta \backslash I \rb = 2 }} \BZ_I \longrightarrow \ldots \longrightarrow \bigoplus\limits_{\substack{I \subset \Delta \\ \lb \Delta \backslash I \rb = \lb \Delta \rb -1 }} \BZ_I \longrightarrow \BZ_\emptyset \longrightarrow 0
\end{equation}
on the étale site of $Y$ (cf. \cite[Theorem 6.9]{CDHN}). By applying $\Ext^*(\iota_*(-),\CE_\lambda)$ to the complex (\ref{fundamentalcomplexintro}), we get the spectral sequence 
\begin{equation*}\label{spectralintro}
  \hat{E}_1^{-p,q}=\Ext^q(\bigoplus\limits_{\substack{I \subsetneq \Delta \\ \lb \Delta \backslash I \rb = p+1 }} \iota_*(\BZ_I), \CE_\lambda) \Rightarrow \Ext^{-p+q}(\iota_*(\BZ_{Y}),\CE_\lambda)=H^{-p+q}_{Y}(\sF^{\mathrm{ad}},\CE_\lambda). 
\end{equation*}
From this we find a spectral sequence $(E_r, d_r)_{r \geq 1}$ converging to $H^*(\sF^{\mathrm{wa}},\CE_\lambda)$.
We construct a second quadrant double chain complex $D_{\bullet, \bullet}$ (cf. (\ref{double})), where one of the two induced standard spectral sequences coincides with the strong dual of $(E_r, d_r)_{r \in \BN_{\geq 1}}$ and the other one yields the chain complex in Theorem \ref{theorem1intro}. A crucial point in the proof is that 
   $$H^i_{C_I(w)}(\sF,\CE_\lambda)\cong\begin{cases} M_I(w\cdot\lambda) &i=n-l(w),\\ 0 &\text{else} \end{cases}$$
for $I \subset \Delta$ and $w \in W^I$ (cf. Lemma \ref{C_I(w)})\footnote{Here, $C_I(w)$ is a generalized Schubert cell in $\sF$ (cf. Defintion \ref{defC_Iw}) and $M_I(w\cdot\lambda)$ is a generalized parabolic Verma module of highest weight $w\cdot\lambda$ (cf. Example \ref{genVM})}. Because of this fact we can identify each column of the $E_1$-page with the homology of a chain complex which leads to the construction of $D_{\bullet, \bullet}.$\\

For the analysis of (\ref{Cbullet}), we first noted that Orlik und Schraen described the Jordan-Hölder factors of $V^G_B(e)$  (cf. \cite[Theorem 4.6]{OSch}). We partially generalize this to all objects in  $C_\bullet$ (cf. (\ref{Cbullet})). 
\begin{theorem}[Theorem \ref{multiplicities}]\label{multiplicitiesintro}
   \footnote{If the root sytem $\Phi(\bG,\bT)$ has irreducible components of type $B$, $C$, or $F_4$, we assume $p>2$, and if $\Phi(\bG,\bT)$ has irreducible components of type $G_2$, we assume that $p>3$ (cf. \cite[Section 5]{OSt}).\label{fn:note1} } Fix $w,v \in W$ and let $I_0:=I(w)$ respectively $I:=I(v)$ (cf. (\ref{maximalI})). For a subset $J \subset \Delta$ with $J \subset I$, 
   let $v_{P_J}^{P_I}$ be the generalized smooth Steinberg representation of $L_{P_I}$. Then, the multiplicity of the irreducible $G$-representation 
   $\CF^G_{P_I}(L(v\cdot \lambda),v_{P_J}^{P_I})$ in $V^G_B(w)$ is 
   $$ \sum_{\substack{w' \in W  \\ \supp(w')=J \cap I_0}} (-1)^{\ell(w')+\lb J \cap I_0 \rb} m(w'w,v)$$ 
   and we obtain in this way all the Jordan-Hölder factors of $V^G_B(w)$. 
\end{theorem} 

From this we see that the homology of $C_\bullet$ is closely related to the Kazhdan-Lusztig conjecture. Therefore, we compute the Jordan-Hölder factors of the homology of $C_\bullet$ only in concrete examples (cf. Example \ref{exampleCohomology}). In this computations we make use of the fact that the morphism $p_{w',w}:V^G_B(w')\rightarrow V^G_B(w)$ in the differentials of $C_\bullet$ is surjective for $w',w \in W$ with $w' \leq w$ (cf. Lemma \ref{surjection}). \\

The paper is divided into two parts. The first half is devoted the preliminaries that we will use afterwards. In detail, we recall some basics about local cohomology in Section \ref{s:localcohomology}, about a split reductive group $\bG$ over a finite extension of $\BQ_p$ in Section \ref{s:rootdatum}, about  the functor $\CF^G_P$ in Section \ref{s:FGP} and last but not least about $p$-adic period domains in Section \ref{s:perdom}. In 
Section \ref{s:FGP} we also prove Theorem \ref{multiplicitiesintro}. 

In the second half, we first introduce our setup in Section \ref{s:Setup}. 
This includes the period domain $\sF^{\wa}$ inside the complete flag variety 
$\sF$ over $K$ and the line bundle $\CE_\lambda$ on $\sF$ associated to a dominant character $\lambda$ of $\bG$ with respect to the Borel pair $(\bT,\bB)$. 
In the next section, we make some geometric observations for the complement $Y$ of $\sF^{\wa}$ in $\sF^\ad$. In particular, (generalized) Schubert cells and unions of Schubert varieties will appear there. Then, in Section \ref{s:AlgLocCo}, we determine the algebraic local cohomology groups of $\sF$ 
with support in these (locally) closed subsets and coefficients in $\CE_\lambda$. We relate them to analytic local cohomology groups of $\sF^{\rig}$ in Section \ref{s:AnaLocCo}. In Section \ref{s:results}, we use this relation to deduce Theorem \ref{theorem1intro}. For this we use Orlik's fundamental complex (cf. \cite[Section 6.2.2]{CDHN}) and an induced spectral sequence in cohomology. Moreover, we determine the Jordan-Hölder factors of the dual of $H^*(\sF^{\wa}, \CE_\lambda)$ in examples with the help of the computer. 

In the Appendix \ref{s:Appendix} we list the code we use to determine the Jordan-Hölder factors of the terms in the chain complexes in the examples given. 


\subsection{Notations}
 Let $p$ be a prime and $K$ be a finite extension of $\BQ_p$. Further let $L$ be a complete extension of $\BQ_p$ with $K\subset L$. We let $\CO_K$ and $\CO_L$, respectively, be the ring of integers of $K$ and $L$, respectively. Moreover, let $\lb \text{\,\,\,} \rb$ be the absolute value of $K$ and $L$, respectively, such that $\lb p \rb = p^{-1}$. \\


We use bold letters for algebraic group schemes over $K$, e.g. $\bG$, $\bB$. The corresponding groups of $K$-valued points are denoted by normal letters, e.g.  $G$, $B$ and the associated Lie algebras by Gothic letters, e.g. $\fkg$, $\fkb$. We write $U(\fkh)$ for the universal enveloping algebra of a Lie algebra $\fkh$ over $K$. \\

We consider $L$ as the field of coefficients. The base change of a $K$-vector 
space or a scheme over $K$ to $L$ is indicated by $L$ in the subscript, e.g. $\mathfrak{g}_L=\mathfrak{g}\otimes_K L$. 
We make an exception when considering a universal enveloping algebra, i.e. we will write  $U(\fkh)$ for $U(\fkh)_L \cong U(\fkh_L)$. \\

We denote by $\Rep^{\infty,\, \mathrm{adm}}_L(H)$ the category of smooth admissible representations of a locally profinite group $H$ on $L$-vector spaces, as in \cite[Section 2.1]{BH2}. \\

For a locally convex $L$-vector space $V$, we denote by $V'$ the strong dual, i.e. the $L$- vector space of continuous linear forms equipped with
the strong topology of bounded convergence. \\ 

For an algebraic variety $X$ over $K$, we write $X^\mathrm{rig}$ for the rigid analytic variety and by $X^\ad$ the adic space attached to $X$, respectively. If $\CE$ is a sheaf on such a variety $X$, we also write $\CE$ for the associated sheaf on $X^{\rig}$, $X^{\ad}$ and its restriction to any subspace, respectively.

\subsection{Acknowledgements} This paper is part of a PhD thesis of the Bergische Universität Wuppertal. My biggest thanks therefore goes to Sascha Orlik. He gave me the opportunity for this research project. It is a multifaceted topic, which somehow miraculously connects areas of mathematics that at first glance do not seem to be in contact with each other - wonderful. Moreover, he actively mentored me throughout the process. \\

I would also like to thank Roland Huber, Thomas Hudson, Henry July, Georg Linden and Sean Tilson for all the helpful mathematical conversations and the support throughout the process. \\

This research was conducted in the framework of the research training group GRK 2240: Algebro-Geometric Methods in Algebra, Arithmetic and Topology, which is funded by the DFG.
\section{Preliminaries}\label{s:preliminaries} 

\subsection{Local Cohomology}\label{s:localcohomology} 
In this subsection, we recall some facts about the local cohomology of topological spaces. For that purpose, we follow the theory described in \cite[Section 1]{H}. Let $X$ be a topological space and $Z \subset X$ a locally closed subset, i.e. there is an open subset $V$ such that  $Z$ is closed in $V$. Further, let $\CE \in \mathbf{Ab}(X)$ be an abelian sheaf on $X$. Then, $H^*_Z(\sF,\CE)$ be the right derived functors of $$\Gamma_Z(X,\CE):=\Ker\big(\Gamma(V,\CE)\rightarrow \Gamma(V\backslash Z,\CE)\big).$$ An essential property that we take advantage of is the following.


\begin{proposition}\label{excision}\cite[Proposition 1.3]{H}
    Let $Z$ be locally closed in $X$, and let $V$ be open in $X$ and such that $Z \subseteq V \subset X$. Then, for any $\CE \in \mathbf{Ab}(X)$, 
    $$ H^i_Z(X,\CE) \cong H^i_Z(V,\CE|_V).$$
\end{proposition}
For two closed subsets $Z_1,Z_2 \subset X$ with $Z_1 \subset Z_2$, we let $$\Gamma_{Z_1/ Z_2}(X, \CE):=\Gamma_{Z_1}(X, \CE)/\Gamma_{Z_2}(X, \CE).$$ 
Then, $H^*_{Z_1/ Z_2}(X, \CE)$ denotes the right derived functor of $\Gamma_{Z_1/ Z_2}(X, -)$ as defined in \cite[Section 7, p. 349/350]{K}. It comes with the following property. 

\begin{lemma}\label{locrel}\cite[Lemma 7.7]{K}
    Let $Z_1 \supset Z_2$ be two closed subsets of a topological space $X$. Let $\CE$ be any abelian sheaf on $X$. Then, there is a natural isomorphism 
    $$ H^i_{Z_1/Z_2}(X,\CE) \longrightarrow H^i_{Z_1\backslash Z_2}(X\backslash Z_2, \CE)$$
    for all integers $i$.  
\end{lemma}

Next, we state a rather technical result which will be helpful afterwards. It seems somehow standard as it is for example used in \cite{MR} and \cite[Section 2.1]{FKT}, but we could not find a precise statement fitting our purposes.
    
\begin{lemma}\label{cousinspectral}
Let $X \supset Z_0 \supset Z_1 \supset \ldots  \supset Z_n$ be a filtration of $X$ by closed subsets and $\CE$ an abelian sheaf on $X$. Then, there is a 
spectral sequence 
$$E_1^{pq}= H^{p+q}_{Z_p/Z_{p+1}}(X, \CE)\Rightarrow H^{p+q}_{Z_0}(X,\CE),$$
where the morphisms on the $E_1$-page are the natural ones. 
\end{lemma}
\begin{proof}
We associate to $\CE$ a complex of flasque sheaves $\CG^\bullet(\CE)$ on $X$ together with an augmentation map $\CE \rightarrow \CG^0(\CE)$ such that 
$\CE \rightarrow \CG^\bullet(\CE)$ is a resolution of $\CE$ (cf. \cite{G}). As Kempf pointed out in \cite[Section 7, p. 350]{K}, we can use this resolution to compute the local cohomology groups in question. \\

The given filtration on $X$ naturally defines a filtration of complexes  
$$\Gamma_{Z_0}(X,\CG^\bullet(\CE)) \supset \Gamma_{Z_1}(X,\CG^\bullet(\CE)) \supset \ldots \supset \Gamma_{Z_n}(X,\CG^\bullet(\CE)) $$
from which we form the following quotient complexes 
\begin{equation}\label{quotient}
0 \rightarrow \Gamma_{Z_{j+1}}(X,\CG^\bullet(\CE)) \rightarrow \Gamma_{Z_j}(X,\CG^\bullet(\CE)) \rightarrow K_j^{\bullet} \rightarrow 0.
\end{equation}
Notice that by \cite[Section 7]{K}, one has $K_j^\bullet=\Gamma_{Z_j/Z_{j+1}}(X, \CG^\bullet(\CE))$ such that $$H^q(K_j^\bullet)=H^q_{Z_j/Z_{j+1}}(X, \CE).$$
Then, by the procedure explained in \cite[Section 3]{Ri}, we get an exact complex of complexes 
\begin{equation}\label{dcoffc}
0 \rightarrow \Gamma_{Z_0}(\sF,\CG^\bullet(\CF)) \rightarrow \tilde{K}_0^\bullet \rightarrow \tilde{K}^\bullet_1[1] \rightarrow \tilde{K}^\bullet_2[2] \rightarrow \ldots \rightarrow \tilde{K}^\bullet_n[n] \rightarrow 0
\end{equation}
where $\tilde{K}^\bullet_j$ is a complex quasi-isomorphic to $K^\bullet_j$. Moreover, the morphisms $\tilde{K}^\bullet_j \rightarrow \tilde{K}^\bullet_{j+1}[1]$ induces the natural homomorphisms
$H^q_{Z_j/Z_{j+1}}(X,\CE)\rightarrow H^{q+1}_{Z_{j+1}/Z_{j+2}}(X,\CE) $ which are given by the connecting homomorphisms coming from the long exact sequence in cohomology of (cf. (\ref{quotient})) followed by the quotient maps.\\

Thus, (\ref{dcoffc}) yields a double complex 
$$C^{\bullet, \bullet}:\tilde{K}_0^\bullet \rightarrow \tilde{K}^\bullet_1[1] \rightarrow \tilde{K}^\bullet_2[2] \rightarrow \ldots \rightarrow \tilde{K}^\bullet_n[n].$$ 
Then, by combining the properties mentioned above with the usual theory of spectral sequences associated to a double complex, we obtain the result we were looking for. 
\end{proof}

\subsection{Split reductive groups}\label{s:rootdatum} 
We recall the basic facts that come along with a split connected reductive group which will be essential throughout the whole thesis. \\


Let $K$ be a finite extension of $\BQ_p$ and $\bG$ a split connected reductive group over $K$. 
Any split maximal torus $\bT \subset \bG$ of rank $d$ defines the \textit{split pair} $(\bG,\bT)$ \textit{of rank $d$} to which we associate the root datum 
$$(X^*(\bT),\Phi(\bG,\bT), X_*(\bT),\Phi^\vee(\bG,\bT))$$ with the natural pairing 
\begin{align}
   \langle \text{ , } \rangle:X^*(\bT)_\BQ \times X_*(\bT)_\BQ &\longrightarrow \mathbb{\BQ}  \label{rootpairing}
\end{align}
(cf. \cite[Part II, 1.13]{J}). Furthermore, there exists an \textit{invariant inner product} on $\bG$ due to Totaro, abbreviated by $\inva$ (cf. \cite[Section 5.2.1]{CDHN}). That means 
we have a non-degenerate positive definite symmetric bilinear form $(\text{ , })$ on $X_*(\bT)_\BQ$ for all maximal tori $\bT$ (defined over $\overline{K}$) of $\bG$ such that 
the maps 
\begin{equation*} 
   X_*(\bT)_{\BQ} \rightarrow X_*(g \bT g^{-1})_{\BQ}, \, 
   X_*(\bT)_{\BQ} \rightarrow X_*(\tau \bT \tau^{-1})_{\BQ},
\end{equation*}
are isometries for all $g \in \bG(\overline{K})$ and $\tau \in \Gal(\overline{K}/K)$. A chosen $\inva$ on $\bG$, for any split pair $(\bG,\bT)$, together with the natural pairing (\ref{rootpairing}) induces an isomorphism of $\BQ$-vector spaces 
\begin{align*}
   X^*(\bT)_\BQ &\longrightarrow X_*(\bT)_\BQ,\\
   \chi &\mapsto \chi^*, 
\end{align*}
such that 
\begin{equation}\label{products} 
   (\chi^*, \mu)=\langle \chi, \mu \rangle
\end{equation} 
for all $\mu \in X_*(\bT)$. Furthermore, for $\alpha \in \Delta$ we have (cf. \cite[Ch. VI, §1.1., Lemma 2]{B})
\begin{equation}\label{dual}
   \alpha^{\vee}=\frac{2\alpha^*}{(\alpha^*,\alpha^*)}.
\end{equation}

For the rest of the subsection, we fix a split maximal torus $\bT$ of $\bG$ and an 
$\inva$ $(\text{ , })$  on $\bG$. Using $\Phi:=\Phi(\bG,\bT)$ for the root system and fixing a Borel subgroup $\bB$ inside $\bG$ containing $\bT$, we get a set of corresponding positive 
roots $\Phi^+ \subset \Phi$ and simple roots $\Delta \subset \Phi^+$ as explained in \cite[Section 16.3.1]{Sp}. We call such a tuple $(\bT, \bB)$ a \textit{Borel pair}. \\

There is the following relation between simple roots and coroots. 
\begin{lemma}\label{rootrel}\cite[Lemma 8.2.7]{Sp} For $\alpha, \beta \in \Delta$, $\alpha \neq \beta$, we have  $\langle \alpha, \beta^\vee \rangle \leq 0$.
\end{lemma}

Let $W=N_\bG(\bT)/\bT$ be the Weyl group of $\bG$ with longest element $w_0$ with respect to $\bB$. The natural action of $W$ on $\bT$ by conjugation induces an action on $X^*(\bT)$. We denote by $S:=\{s_\alpha\}_{\alpha \in \Delta} \subset W$, the \textit{simple reflections}, a set of generators of $W$ such that 
\begin{equation*}
   s_\alpha.\lambda=\lambda - \langle \lambda, \alpha^\vee \rangle \alpha 
\end{equation*}
for $\lambda \in X^*(\bT)_\BQ$ and $\alpha^\vee \in \Phi^\vee(\bG,\bT)$. 
For $w \in W$, we consider also the \textit{dot action} given by 
\begin{equation*} \label{dotaction} w \cdot \lambda=w.(\lambda+\rho)-\rho
\end{equation*}
where $\rho:=\frac{1}{2}\sum_{\alpha \in \Phi^+} \alpha $. The support $\supp(w)$ of an element $w \in W$ is the set of simple reflections contained in a (thus in any) reduced expression of $w$. \\

Each $I \subset \Delta$ defines a root system $\Phi_\RI \subset \Phi$ with positive 
roots $\Phi_\RI^+ \subset \Phi_\RI$ and a Weyl group $W_\RI \subset W$ generated by the $\{s_\alpha\}_{\alpha \in I}$ (cf. \cite[Part II, 1.7]{J}). 
We denote by $W^I$ the right \textit{Kostant representatives}, i.e.  the set of minimal length right coset representatives in $W_I\backslash W$. It can be described as (cf. \cite[(2.2)]{B1})
$$W^I=\{w \in W \mid l(s_\alpha w)>l(w) \text{ for all } \alpha \in I\}.$$ 
\begin{lemma}\label{Kostant}\cite[Section 0.3 (4)]{H2}
Let $w \in W$ and $I \subset \Delta$. Then, $w \in W^I$ if and only if $w^{-1}\alpha \in \Phi^+$  for all $\alpha \in \Phi_I^+$. 
\end{lemma}
For $w \in W$, we let 
\begin{equation}\label{maximalI}
   I(w):=\{\alpha \in \Delta \mid l(s_{\alpha }w)>l(w)\} \subset \Delta
\end{equation}
be the unique maximal subset such that $w \in W^{I(w)}$ (cf. \cite[p. 663]{OSt}). We have an inclusion preserving bijection (cf. \cite[Proposition 12.2]{MT}) 
\begin{align}\label{parabolicbijection}
    \CP(\Delta) &\longleftrightarrow   \{\text{parabolic subgroups  } \bP\supset \bB  \} \\
    I &\mapsto \bB W_I\bB =:\bP_I \nonumber
\end{align} 
where the subgroups $\bP_I$ denote the \textit{standard parabolic subgroups} of $\bG$ with respect to $\bB$, e.g. $\bP_\emptyset=\bB$, $\bP_\Delta=\bG$. Furthermore, each $\bP:=\bP_I$ admits a \textit{Levi decomposition} $\bP=\mathbf{L_P}\cdot \mathbf{U_P}$ (cf. \cite[Part II,1.8]{J}). 
Here, $\mathbf{L_P}$ denotes the standard \textit{Levi factor} containing $\bT$ and $ \mathbf{U_P}$ the \textit{unipotent radical} of $\bP$. Additionally, we let $\mathbf{U}^-_{\bP}$ be its \textit{opposite unipotent radical}. 
\begin{remark}\label{intersectionparabolic}
   Let $I,J \subset \Delta$. Since $W_I \cap W_J=W_{I \cap J}$, one sees that $\bP_I \cap \bP_J=\bP_{I \cap J}$. 
\end{remark}

We define 
\begin{equation}\label{dominantweights}
   X^*(\bT)_I^+:=\Big\{\lambda \in X^*(\bT)\, \Big\vert \, \langle \lambda, \alpha^{\vee} \rangle \geq 0 \text{ for all } \alpha \in I \Big\}
\end{equation}
for $I \subset \Delta$  to be the set of \textit{$\bL_{\bP_{I}}$-dominant weights}. For $I=\Delta$, we just write $ X^*(\bT)^+$ and call it the set of \textit{dominant weights}. 

\begin{proposition}\label{parabolicweight}\cite[p. 502]{Le1}
   Let $\lambda \in X^*(\bT)^+$, $w \in W$ and  $I \subset \Delta$. If $w \in W^I$, then $w\cdot\lambda \in X^*(\bT)_I^+$.
\end{proposition}

\begin{remark}
   If $\lambda$ is additionally regular, i.e. $\langle \lambda +\rho, \alpha^\vee \rangle \neq 0$ for all $\alpha \in \Phi$ (cf. \cite[Section 1.8]{H2}), then also the converse holds (cf. \cite[Proposition 2.4]{B1}).
\end{remark}

The derived group $\bG_{\der}$ is a connected semi-simple subgroup of $\bG$ with maximal torus 
\begin{equation}\label{derivedtorus}
    \bT_\der:=\langle \Im(\alpha^\vee) \mid \alpha \in \Phi \rangle \subset \bT 
\end{equation}
(cf. \cite[Proposition 8.1.8/ Section 16.2.5]{Sp}). Moreover, $\bT_\der$ splits by \cite[Proposition 8.2.(c)]{Bor}. The natural map 
\begin{align} 
           X^*(\bT)&\longrightarrow X^*(\bT_\der) \label{torusmap}\\ 
            \lambda &\mapsto \lambda\circ \iota, \nonumber
\end{align}
induced by the inclusion $\iota$ (\ref{derivedtorus}), is injective after restriction to $\Phi$ (cf. \cite[Section 8.1, p. 135]{Sp}). Thus, we identify $\Phi$ with its image. Therefore, the split pair $(\bG_\der, \bT_\der)$ has the root datum 
$$(X^*(\bT_\der), \Phi, X_*(\bT_\der), \Phi^\vee)$$ (cf. \cite[Corollary 8.1.9]{Sp}) and we denote the associated pairing by $\langle \text{ , } \rangle_\der$.   
By semi-simplicity of $\bG_\der$, the simple roots $\Delta$ form a basis of $X^*(\bT_\der)_\BQ$ (cf. \cite[Part II, 1.6]{J}). Thus, we define
the dual basis 
\begin{equation}\label{dualbase}
   \{\varpi_\alpha \mid \alpha \in \Delta \} \subset X_*(\bT_\der)_\BQ,
\end{equation} i.e. $\langle \beta, \varpi_\alpha \rangle_\der = \delta_{\alpha,\beta}$ for all $\alpha, \beta \in \Delta$. Naturally, $\{\varpi_\alpha \}_{\alpha \in \Delta} \subset X_*(\bT)_\BQ$. 
By duality, the corresponding set of coroots $\{\alpha^\vee \mid \alpha \in \Delta\}$ forms a basis of $X_*(\bT_\der)_\BQ$ (cf. \cite[Part II, 1.6]{J}) and we analogously define the dual basis 
$$\{\check{\varpi}_\alpha \mid \alpha \in \Delta \} \subset X^*(\bT_\der)_\BQ$$
whose elements are known as \textit{fundamental weights}. A helpful observation for later is the following. 

\begin{lemma}\label{equivalentcond}
Let $\mu = \sum_{\alpha \in \Delta} n_\alpha \alpha^\vee \in X_*(\bT_\der)_\BQ \subset X_*(\bT)_\BQ$ with $n_\alpha \in \BQ$, and $\beta \in \Delta$. Then, $$(\mu, \varpi_\beta)>0 \text{ if and only if } \langle \check{\varpi}_\beta, \mu \rangle_\der >0.$$ 
\end{lemma}
\begin{proof} 
Let $\alpha \in \Delta$. We have by (\ref{dual}) and (\ref{products})
   $$ (\alpha^{\vee}, \varpi_\beta)=\big(\frac{2}{(\alpha^*,\alpha^*)}\alpha^*,\varpi_\beta\big)= \frac{2}{(\alpha^*,\alpha^*)} \langle  \alpha, \varpi_\beta \rangle.$$ 
As the natural pairings are induced by the composition of a cocharacter with a character (cf. \cite[Part II, Section 1.3]{J}), we see that 
$$\langle  \alpha, \varpi_\beta \rangle=  \langle  \alpha, \varpi_\beta \rangle_\der.$$ Thus, 
$$ (\alpha^{\vee}, \varpi_\beta) = \frac{2}{(\alpha^*,\alpha^*)} \langle  \alpha, \varpi_\beta \rangle_\der=\frac{2}{(\alpha^*,\alpha^*)}\delta_{\alpha,\beta}.$$
Hence, 
\begin{equation*}
(\sum n_\alpha \alpha^{\vee}, \varpi_\beta )=\frac{2}{(\beta^*, \beta^*)} n_\beta> 0\text{ if and only if } \langle \check{\varpi}_\beta , \sum n_\alpha \alpha^{\vee} \rangle_\der= n_\beta>0. \qedhere
\end{equation*}
\end{proof}

Since the fundamental weights form a basis of $X^*(\bT_\der)_\BQ$, we notice that 
\begin{equation}\label{linearcombi}
   \alpha = \sum_{\beta \in \Delta} \langle \alpha, \beta^\vee{}\rangle_\der \check{\varpi}_\beta
\end{equation}
for $\alpha \in \Delta$.  After we fix an ordering on $\Delta=\{\alpha_1 < \alpha_2 < \ldots <  \alpha_r \}$, the \textit{Cartan matrix} is defined as 
\begin{equation}\label{cartanmatrix}
C \in \BQ^{\lb \Delta \rb \times \lb \Delta \rb } \text{ with } C_{ji}:=\langle \alpha_i, \alpha_{j}^\vee  \rangle_\der.
\end{equation}
Hence, by (\ref{linearcombi}), it is the base change matrix from $\{\alpha\}_{\alpha \in \Delta}$ to $\{\varpi_\alpha\}_{\alpha \in \Delta}$. For the inverse of $C$, we will need the following fact.
\begin{lemma}\label{inversecartan}
   Let $\Phi$ be irreducible. Then, all entries of $C^{-1}$ are positive rational numbers. 
\end{lemma}
\begin{proof}
   This is explained in \cite[Section 5,p. 19]{LG}.
\end{proof}
 With the definition of (\ref{dualbase}), we also obtain an alternative description of the standard parabolic subgroups of $\bG$. For a one-parameter subgroup $\mu \in X_*(\bG)$ defined over some field extension $L$ of $K$, we denote by $\bP(\mu)$ the parabolic subgroup of $\bG_L$ whose $\overline{K}$-valued points are given by 
\begin{equation*}
   \bP(\mu)(\ov K)=\big\{g \in \bG(\ov K)  \mid \lim_{t \rightarrow 0} \mu(t)g\mu(t)^{-1} \text{ exists in } \bG(\ov K) \big\}
\end{equation*}
(cf. \cite[Definition 2.3/Proposition 2.6]{M}). We have seen that $\langle \beta, \varpi_\alpha\rangle= \delta_{\alpha,\beta}$ for $\alpha, \beta \in \Delta$. Thus, \cite[Proof of Proposition 8.4.5/Lemma 15.1.2]{Sp} implies that $\bP_{\Delta\backslash\{\alpha\}}=\bP(\varpi_\alpha)$. We deduce from Remark \ref{intersectionparabolic} that 
\begin{equation*}
   \bP_I=\bigcap_{\alpha \not\in I}\bP_{\Delta\backslash\{\alpha\}}=\bigcap_{\alpha \not\in I}\bP(\varpi_\alpha)
\end{equation*}
for $I \subset \Delta$. \\

\subsection{The functor $\CF^G_P$}\label{s:FGP}

Let the ground field $K$ be a finite extension of $\BQ_p$ and $(\bG,\bT)$ a split pair over $K$ of rank $d$ (cf. section \ref{s:rootdatum}). Further, let $(\bT, \bB$) be a fixed Borel pair and $L$ a finite extension of $K$. Let $G=\bG(K)$ and $P=\bP(K)$ for some standard parabolic subgroup $\bP$ of $\bG$. As mentioned in \cite[p. 443]{ST2}, $G$ and $P$ are locally $K$-analytic groups.  \\

Orlik and Strauch defined in \cite{OSt} the bi-functor 

\begin{align*}\CF^G_P: \CO^{\fkp}_{\alg} \times \Rep^{\infty,\mathrm{adm}}_L(L_P) &\longrightarrow \Rep^{\ell a}_{L}(G)  \\
(M,V) &\mapsto \CF^G_P(M,V)
\end{align*}
which is contravariant in $M$ and covariant in $V$ (cf. \cite[Proposition 4.7]{OSt}). In case $V$ is the trivial representation $\textbf{1}$, we will write $\CF^G_P(M)$. We briefly explain the basics about this functors and start with the mentioned categories. \\

Let $V$ be a Hausdorff barelled locally convex $L$-vector space. Then, $C^{an}(G,V)$ is the \textit{locally convex $L$-vector space of locally $L$-analytic functions on $G$ with values in $V$} (see \cite[Section 2, p. 447]{ST2} for a detailed description). 
Further, $$D(G):=C^{an}(G,L)'$$ is the \textit{locally convex vector space of $L$-valued distributions on G} (cf. \cite[Section 2, Definition, p. 447]{ST2}). Additionally, with convolution as multiplication, it is an associative $L$-algebra  (cf. \cite[Proposition 2.3]{ST2}).

\begin{definition}\cite[Section 3, p. 451, Definition]{ST2} A \textit{locally analytic $G$-representation} $V$ (over $L$) is a Hausdorff barelled locally convex $L$-vector space $V$ equipped with a $G$-action by continuous linear endomorphisms such that, for each $v \in V$, the orbit map $\rho_v(g):=gv$ lies in $C^{an}(G,V)$.    We denote the category of such representations by $\Rep^{\ell a}_{L}(G)$.

\end{definition}
As in the algebraic or smooth case,   we also have the induction functor. 
\begin{definition}\cite[Section 2.2, p. 103]{OSt}
    Let $H$ be a closed subgroup of $G$ and $(V,\rho)$ a locally analytic representation of $H$. The \textit{locally analytic induced representation $\Ind^G_H(V)$} is defined as 
$$  \Ind^G_H(V)=\{f \in C^{an}(G,V) \mid  f(gh)=\rho(h^{-1})f(g) \, \forall h \in H, \forall g \in G \}.$$
The group $G$ acts on $\Ind^G_H(V)$ by $(g.f)(x)=f(g^{-1}x)$. 
\end{definition}

Over the complex numbers, the BGG category $\mathcal{O}$ and its parabolic version $\mathcal{O}^{\mathfrak{p}}$ provide powerful tools to investigate (infinite dimensional) representations of Lie algebras. A good reference for this topic is \cite{H2}. For our setting, this was considered in detail in \cite[Section 2.5]{OSt} by Orlik and Strauch.

\begin{definition}\cite[p. 106]{OSt} \label{catOp}
   By $\mathcal{O}^{\mathfrak{p}}$ we denote the full subcategory of 
   $\operatorname{Mod}U(\mathfrak{g})$ whose objects $M$ satisfy  the following conditions: 
\begin{enumerate}[leftmargin=11mm]
    \item[($\mathcal{O}^\fkp1$)] $M$ is a finitely generated $U(\fkg)$-module. 
    \item[($\mathcal{O}^\fkp2$)] Viewed as an $\fkl_{\RP,L}$-module, $M$ is the direct sum of finite dimensional simple modules.
    \item[($\mathcal{O}^\fkp3$)]$M$ is locally $\fku_{\RP,L}$-finite.
   \end{enumerate}       
\end{definition}

If $\bP$ is a Borel, we denote the category by $\mathcal{O}$. Notice  that $\mathcal{O}^{\fkg}$ is the 
category of all finite dimensional (semisimple) $U(\mathfrak{g})$-modules. Moreover, for a standard parabolic $\mathbf{Q} \supset \mathbf{P} $, we have that $\mathcal{O}^{\mathfrak{q}} \subset 
\mathcal{O}^{\mathfrak{p}}$. Hence, $\mathcal{O}^{\mathfrak{p}}$ is a full subcategory of $\mathcal{O}$ 
and contains all finite dimensional  $U(\mathfrak{g})$-modules. Additionally, $\mathcal{O}^{\mathfrak{p}}$ 
is an $L$-linear, abelian, artinian and noetherian category which is closed under taking submodules and quotients. 
Thus, the Jordan-Hölder series of an object of $\mathcal{O}^{\mathfrak{p}}$ lies in $\mathcal{O}^{\mathfrak{p}}$. \\

Letting $\operatorname{Irr}(\mathfrak{l}_{P,L})^{\operatorname{fd}}$ be the set of 
isomorphism classes of finite dimensional irreducible $\mathfrak{l}_{P,L}$-modules, we have for $M \in \mathcal{O}^{\mathfrak{p}}$ that 
\begin{equation*}
    M=\bigoplus_{\mathfrak{a} \in \operatorname{Irr}(\mathfrak{l}_{P,K})^{\operatorname{fd}}} 
    M_\mathfrak{a},
\end{equation*}
by property ($\mathcal{O}^\fkp2$) in Definition \ref{catOp}, with $M_\mathfrak{a} \subset M$ being the $\mathfrak{a}$-isotypic part of the representation 
$\mathfrak{a}$. Similiar to before there is an algebraic subcategory in $\mathcal{O}^{\mathfrak{p}}$.

\begin{definition}\cite[p. 106]{OSt}  Let $\mathcal{O}^{\mathfrak{p}}_{\mathrm{alg}}$ be the full subcategory of 
$\mathcal{O}^{\mathfrak{p}}$ with objects $M \in \mathcal{O}^{\mathfrak{p}}$ satisfying the following property: 
whenever $M_\mathfrak{a}\neq 0$, then,  $\mathfrak{a}$ is induced by a finite dimensional algebraic 
$\bL_{\bP,L}$-representation. 
\end{definition}

\begin{definition}\cite[Section 1.15]{H2}
   Let $M \in \CO$. The \textit{formal character} of $M$  is defined as 
   \begin{align*}
      \ch(M):\fkt^*_L &\longrightarrow \BZ^+  \\
             \lambda &\mapsto \dim_L(M_\lambda).
   \end{align*}
\end{definition}

Let $M \in \CO^{\fkp}_{\alg}$. Then, by the very definition of the category $\CO^{\fkp}_{\alg}$, there is a finite-dimensional representation $(W,\rho) \subset M$ of $\fkp_L$ which generates $M$ as $U(\fkg)$-module. We call such a tuple $(M,W)$ an \textit{$\CO^{\fkp}_{\alg}$-pair}. Hence, such a pair comes with a short exact sequence of $U(\fkg)$-modules
\begin{equation}\label{sesCatO} 
   0 \rightarrow \fkd \rightarrow U(\fkg) \otimes_{U(\fkp)} W \rightarrow M \rightarrow 0
\end{equation}
with $\fkd$ being the kernel of the natural map $ U(\fkg) \otimes_{U(\fkp)} W \rightarrow M$.

\begin{example}\cite[Example 2.10]{OSt} \label{genVM} Let $I \subset \Delta$ such that $\bP=\bP_I$.  For  $\lambda \in X^*(\bT)_I^+$, there is a corresponding finite dimensional 
   irreducible algebraic $\bL_{\bP,L}$-representation $V_I(\lambda)$, 
   which can be viewed as a $\bP_L$-representation by letting $\bU_{\bP,L}$ act trivially on it. Then, 
   $$M_I(\lambda)=U(\fkg) \otimes_{U(\fkp)} V_I(\lambda)$$
   is the \textit{generalized parabolic Verma module} associated to $\lambda$, with simple quotient $L(\lambda)$ which both lie in $\CO^{\fkp}_{\alg}$. In case $I=\emptyset$ we omit the subscript. We have a surjection
               $$q_I:M(\lambda) \rightarrow M_I(\lambda)$$ 
   with the kernel being the image of $\bigoplus_{\alpha \in I} M(s_{\alpha} \cdot \lambda) \rightarrow M(\lambda)$ (cf. \cite[Proposition 2.1]{Le2}).
   Furthermore, for $J \subset I$, there is a transition map 
   \begin{equation}\label{transition}
     q_{J,I}:M_J(\lambda) \rightarrow M_I(\lambda)
   \end{equation}
  such that $q_I=q_{I,J}\circ q_J$ (cf. \cite[Section 2, p. 653]{OSch}).
\end{example}

 \begin{lemma}\label{Verma}\cite[Lemma 1]{B}
Let $\lambda, \mu \in \fkt_L^*$ and $M$ a $U(\fkg)$-submodule of $M(\lambda)$, such that $\ch(M)=\ch(M(\mu))$. Then, $M$ is isomorphic to $M(\mu)$. 
\end{lemma}

Next, we say a few words about the construction of the functor $\CF^G_P$. For this, let $V \in \Rep^{\infty,\, \mathrm{adm}}_L(L_P)$. By inflation, we consider $V$ as a representation of $P$. Equipping $V$ with the finest locally convex  $L$-vector space topology, it is of compact type and carries the structure of a locally analytic $P$-representation (cf. \cite[p. 117]{OSt}). For an $\CO^{\fkp}_{\alg}$-pair $(M,W)$, Orlik und Strauch consider $W'\otimes_L V$ as the projective (or inductive) tensor product which is complete and a locally analytic $P$-representation via the diagonal action (cf. \cite[p. 117]{OSt}). Then, they defined\footnote{\begin{equation*}
   \langle \, , \, \rangle_{C^{an}(G,V)}:D(G)\otimes_{D(P)} W \times \Ind^G_P(W'\otimes V) \longrightarrow C^{an}(G,L), \,
   (\delta \otimes w, f)\mapsto \Big[g \mapsto \Big(\delta \cdot_r\big(ev_w \circ f \big)\Big)(g)\Big]     
   \end{equation*}} (cf. \cite[(4.4.1), p. 117]{OSt})
\begin{align*}
   \CF^G_P(M,W,V)&:= \Ind^G_P(W'\otimes V)^\fkd \\
                 &:=\{f \in \Ind^G_P(W'\otimes V) \mid \langle \delta  , f \rangle_{C^{an}(G,V)} =0 \text{ for all } \delta \in \fkd \}.
\end{align*}

The definition is independent of the chosen $\CO^{\fkp}_{\alg}$-pair $(M,W)$ (cf. \cite[Section 4.6]{OSt}). Therefore, we write $\CF^G_P(M,V)$ for any $\CF^G_P(M,W, V)$. One of the most basic properties of the functor $\CF^G_P$ is the following. 
\begin{proposition}\label{exact}\cite[Proposition 4.9]{OSt} \begin{enumerate}[label=\roman*)]
   \item The bi-functor $\CF^G_P$ is exact in both arguments. 
   \item If $Q \supset P$ is a parabolic subgroup, $\fkq=\Lie(Q)$, and $M \in \CO^{\fkq}_{\alg}$, then 
               $$\CF^G_P(M,V)= \CF^G_Q\big(M,i^{L_Q}_{L_P(L_Q\cap U_P)}(V)\big)$$
   where $i^{L_Q}_{L_P(L_Q\cap U_P)}(V)=i^Q_P(V)$ denotes the corresponding induced representation in the category of smooth representations. 
\end{enumerate}
\end{proposition}

We devote the last part of this subsection to an application of the functor $\CF_P^G$. Let $I \subset \Delta$ and $P_I$ be the associated parabolic subgroup of $G$. Then, we notice that 
$$ \Ind^G_{P_I}(\textbf{1})=\CF_{P_I}^G(M_I(0))$$ 
where $0$ is the weight sent to the zero vector under the identification $X^*(\bT)\cong \BZ^d$. More generally, for $\lambda \in X^*(\bT)^+$ (cf. (\ref{dominantweights})), we have 
\begin{equation}\label{IGP}
   I^G_{P_I}(\lambda):=\Ind^G_{P_I}(V_I(\lambda)')=\CF_{P_I}^G(M_I(\lambda))
\end{equation}
since the $\CO_{\alg}^{\fkp_I}$-pair $(V_I(\lambda), M_I(\lambda))$ has trivial kernel $\fkd$ (cf. (\ref{sesCatO})). For $I \subset J \subset \Delta$, the morphism $q_{I,J}$ (cf. (\ref{transition})) induces by the functoriality of the functor $\CF^G_P$ a map 
$$ p_{J,I}:  I^G_{P_J}(\lambda)=\CF_{P_J}^G(M_J(\lambda), \textbf{1}) \xrightarrow{\CF^G_{P_J}(q_{I,J}, \, \mathrm{incl.})} \CF_{P_J}^G(M_I(\lambda), i^{P_J}_{P_I})\cong\CF_{P_I}^G(M_I(\lambda))=I^G_{P_I}(\lambda)$$
of locally analytic $G$-representations. Furthermore, the map $p_{J,I}$ is injective and has closed image (cf. \cite[p. 660]{OSch}). 
\begin{definition}\cite[p. 661]{OSch}
For $I\subset \Delta$, 
$$ V^G_{P_I}(\lambda):= I^G_{P_I}(\lambda)\Big/ \sum_{J \supsetneq I} I^G_{P_J}(\lambda) $$
is the \textit{twisted generalized Steinberg representation} associated to $\lambda$. 
\end{definition}
It has the following resolution in $\Rep^{\ell a}_{L}(G)$. 

\begin{theorem}\cite[Theorem 4.2]{OSch}
Let $\lambda \in X^*(\bT)^+$ and $I \subset \Delta$. Then, the following complex is a resolution of $V^G_{P_I}(\lambda)$ by locally analytic $G$-representations, 
\begin{align*}
   0 \rightarrow I^G_G(\lambda) \rightarrow \bigoplus_{\substack{I \subset K \subset \Delta \\ \lb \Delta \backslash K \rb = 1}} I^G_{P_K}(\lambda) &\rightarrow \bigoplus_{\substack{I \subset K \subset \Delta \\ \lb \Delta \backslash K \rb = 2}} I^G_{P_K}(\lambda) \rightarrow \ldots \\
   \ldots &\rightarrow \bigoplus_{\substack{I \subset K \subset \Delta \\ \lb K \backslash I \rb = 1}} I^G_{P_K}(\lambda) \rightarrow I^G_{P_I}(\lambda) \rightarrow V^G_{P_I}(\lambda)\rightarrow 0. 
\end{align*}
\end{theorem}

Here, the differentials $d_{K',K}: I^G_{P_{K'}}(\lambda) \rightarrow I^G_{P_{K}}(\lambda)$ are defined as follows (\cite[p. 660]{OSch}). We fix an ordering on $\Delta$. Let $K, K' \subset \Delta$ with $\lb K \rb= \lb K' \rb -1$ and $K'=\{\alpha_1 < \ldots < \alpha_r \}$. Then, 
$$ d_{K',K}= \begin{cases} (-1)^i p_{K',K}  &K'=K \cup \{\alpha_i\} \\
                                 0       & K \not\subset K' 
\end{cases}.$$
We like to stress a relative version which was shown in \cite[p. 663]{OSch} in the proof of the previous theorem. For this we follow the notion of \cite[p. 661]{OSch}. 

\begin{definition}\label{twistedSteinberg}
Let $\lambda \in X^*(\bT)^+$, $I \subset \Delta$ and $w \in W^I$. By Proposition \ref{parabolicweight}, we know that $w \cdot \lambda \in X^*(\bT)_I^+$. Then, we set  

\begin{align}
    I_{P_I}^G(w)&:=\Ind^G_{P_I}(V_I(w \cdot \lambda)')=\CF^G_{P_I}(M_I(w \cdot \lambda)) \label{twisted}, \\
    V_{P_I}^G(w)&:=I^G_{P_I}(w)\Big/\sum_{\substack{ J \supsetneq I \\ w \in W^{J}}} I_{P_{J}}^G(w). \nonumber
\end{align}
\end{definition}

\begin{corollary}\label{relativeresolution} Let $\lambda \in X^*(\bT)^+$, $I \subset \Delta$ and $w \in W^I$. Then, the following complex is acyclic 
   \begin{align*}
      0 \rightarrow I^G_{P_{I(w)}}(w) \rightarrow \ldots \rightarrow  \bigoplus_{\substack{I \subset K \subset I(w) \\ \lb K \backslash I \rb = 1}} I^G_{P_K}(w) \rightarrow I^G_{P_{I}}(w) \rightarrow V^G_{P_{I}}(w)\rightarrow 0. 
   \end{align*}
\end{corollary}

In \cite[Theorem 4.6]{OSch}, it was shown that the Jordan-Hölder factors of $V^G_B(\lambda)$ are of the form $\CF^G_{P_I}\Big(L(w\cdot \lambda), v^{P_I}_{P_J}\Big)$ for suitable $I,J \subset \Delta$ and $w \in W$. We will use Corollary \ref{relativeresolution} to get a similar statement for $V^G_B(w)$ which partially generalizes \cite[Theorem 4.6]{OSch}. For $w,v \in W$, we denote by $m(w,v) \in \BZ_{\geq 0}$ the \textit{multiplicity} of $L(v \cdot 0)$ in $M(w \cdot 0)$. It is well known that $m(w,v)>0$ if and only if $ w \leq v$ with respect to the Bruhat order $\leq$ on $W$. The multiplicities can be computed using Kazhdan-Lusztig polynomials (cf. \cite{BB} or \cite{BK}) which is in general only possible in a timely manner with the help of a computer. 

\begin{theorem}\label{multiplicities} \footnote{ 
   We have to assume the following( cf. \cite[Section 5]{OSt}). If the root sytem $\Phi(\bG,\bT)$ has irreducible components of type $B$, $C$, or $F_4$, we assume $p>2$, and if $\Phi(\bG,\bT)$ has irreducible components of type $G_2$, we assume that $p>3$.} Fix $w,v \in W$ and let $I_0:=I(w)$ and $I:=I(v)$, respectively, be as above. For a subset $J \subset \Delta$ with $J \subset I$, 
    let $v_{P_J}^{P_I}$ be the generalized smooth Steinberg representation of $L_{P_I}$. Then, the multiplicity of the irreducible $G$-representation 
    $\CF^G_{P_I}(L(v\cdot \lambda),v_{P_J}^{P_I})$ in $V^G_B(w)$ is 
    $$ \sum_{\substack{w' \in W  \\ \supp(w')=J \cap I_0}} (-1)^{\ell(w')+\lb J \cap I_0 \rb} m(w'w,v)$$ 
    and in this way we obtain all the Jordan-Hölder factors of $V^G_B(w)$. 
\end{theorem} 

\begin{proof} 
    We only have to slightly modify the proof of \cite[Theorem 4.6]{OSch}.  
    From the resolution for $V^G_B(w)$ by Corollary \ref{relativeresolution}, we obtain the multiplicity 
        $$\big[V_B^G(w): \CF^G_{P_I}(L(v\cdot \lambda),v_{P_J}^{P_I})\big]=\sum_{K \subset I_0}(-1)^{\lb K \rb}\big[I^G_{P_K}(w):\CF^G_{P_I}(L(v\cdot \lambda),v_{P_J}^{P_I})\big]$$
   of the simple object $\CF^G_{P_I}(L(v\cdot \lambda),v_{P_J}^{P_I})$ in $V^G_B(w)$. 
    By the arguments mentioned in loc. cit, it follows that $\big[I^G_{P_K}(w):\CF^G_{P_I}(L(v\cdot \lambda),v_{P_J}^{P_I})\big] \neq 0$ if only if  $K \subset J \cap I_0$. In that case we have 
    $$ \big[I^G_{P_K}(w):\CF^G_{P_I}(L(v\cdot \lambda),v_{P_J}^{P_I})\big] = \big[M_K(w \cdot \lambda): L(v\cdot \lambda)\big].$$
    From the character formula 
        $$\ch(M_K(w\cdot \lambda))= \sum_{w' \in W_K} (-1)^{\ell(w')} \ch(M(w'w\cdot \lambda)),$$ 
    (cf. \cite[Section 9.6, p. 189, Proposition]{H2}), we obtain 
    \begin{align*} \big[V_B^G(w): \CF^G_{P_I}(L(v\cdot \lambda),v_{P_J}^{P_I})\big] &= \sum_{K \subset J \cap I_0}(-1)^{\lb K \rb} \sum_{w' \in W_K} (-1)^{\ell(w')} \big[M(w'w\cdot \lambda):L(v \cdot \lambda)\big] \\
     &= \sum_{w' \in W} (-1)^{\ell(w')}\big[M(w'w\cdot \lambda):L(v \cdot \lambda)\big]  \sum_{\substack{K \subset J\cap I_0 \\ \supp(w') \subset K}} (-1)^{\lb K \rb}.
    \end{align*}
    Finally, we have 
    $$ \sum_{\supp(w') \subset K \subset J\cap I_0} (-1)^{\lb K \rb}=(-1)^{\supp(w')}(1-1)^{\lb (J \cap I_0) \backslash \supp(w')\rb}$$
    which is non-zero if and only if $\supp(w')=J \cap I_0$. Hence, the formula follows. \\

    The natural morphism $V^G_B(\lambda)\rightarrow V^G_B(w)$ is surjective for all $w \in W$ as it is induced by an injective morphism $M(w \cdot \lambda) \rightarrow M(\lambda)$ (cf. Lemma \ref{surjection}).  
    Therefore, \cite[Theorem 4.6]{OSch} implies that we obtain all Jordan-Hölder factors of $V^G_B(w)$ in this manner.  
\end{proof}

\subsection{Period domains}\label{s:perdom}
\setenumerate[0]{leftmargin=5mm,itemsep=\the\smallskipamount}

In this section we give a brief introduction to our central object of study,  $p$-adic period domains (cf. \cite{O1} and \cite{CDHN}). 
For a more general setting and detailed presentation, we refer the reader to \cite{DOR}. \\

Let $F$ be an algebraically closed field of characteristic $p$ and $K_0 = \mathrm{Quot}(W(F))$, the quotient field of the ring of Witt vectors of $F$.
Let $K=\BQ_p$ with algebraic closure $\ov{K}$ and absolute Galois group $\Gamma_K=\Gal(\overline{K}/K)$. We denote by $C$ the $p$-adic completion of 
$\ov{K}$. Moreover, let $\sigma \in \Aut(K_0/K)$ be the Frobenius homomorphism
and $\bG$ a quasi-split connected reductive group over $K$.  

\subsubsection{Filtered isocrystals}
An \textit{isocrystal} over $F$ is a pair $(V, \Phi)$ with a finite-dimensional  $K_0$-vector space $V$ and a $\sigma$-linear bijective endomorphism $\Phi$ of $V$. 
Then, an \textit{isocrystal with $\bG$-structure} (also referred to as a $\bG$-isocrystal) is an exact faithful tensor functor $$ \Rep_K(\bG) \longrightarrow \mathrm{Isoc}(F)$$
from the category of finite-dimensional algebraic $K$-representations to the category of isocrystals over $F$. 
 Every $\bG$-isocrystal is 
induced by an element $b \in \bG(K_0)$. Namely, to a finite-dimensional algebraic $K$-representation $(V, \rho)$ of $\bG$, we associate 
$$ N_b(V):=\big(V \otimes_K K_0, \rho(b)(\id_V \otimes\sigma )\big)$$
which defines an isocrystal over $F$  (cf. \cite[Remark  3.4]{RR}). The morphisms are mapped under $N_b$ as expected. Thus, $N_b$ is a $\bG$-isocrystal and 
$b,b' \in \bG(K_0)$ yield the same $\bG$-isocrystal if and only if there exists a $g \in \bG(K_0)$ 
such that $b'=gb\sigma(g)^{-1}$, i.e. if they are $\sigma$-conjugated. The set of $\sigma$-equivalence classes $[b]$ in $\bG(K_0)$ 
is denoted by $B(\bG)$ and was introduced by Kottwitz \cite{Ko1, Ko2}. In \cite[Section 3]{CDHN} the authors give several interpretations of this set. 
Additionally, the $\bG$-isocrystal $N_b$ comes along with its automorphism group $\bJ_b$. It is an algebraic group over $K$ with 
$$\bJ_b(A)=\{g \in \bG(K_0 \otimes_K A) \mid gb\sigma(g)^{-1}=b \}$$ 
for every $K$-algebra A. It depends only on $[b]$ in view of \cite[Section 2.1, p. 280]{RV}. 
As $\bG$ is quasi-split, we know by \cite[Section 6]{Ko1} that $\bJ_b$ is an inner form of a Levi subgroup of $\bG$; hence $\bJ_b$ is reductive. \\

Let $L$ be a field extension of $K_0$. A \textit{filtered isocrystal} $(V, \Phi, \CF^\bullet)$ over $L$ is an isocrystal $(V, \Phi)$ 
over $F$ with a $\BQ$-filtration $\CF^{\bullet}$ (decreasing, exhaustive and separated) on $V_L$. 
The filtered isocrystal over $L$ form a $K$-linear quasi-abelian tensor category $\mathrm{FilIsoc}^L_{F}(\Phi)$ (cf. \cite[Section VIII, p. 192]{DOR}).
Then, we say that a filtered isocrystal $(V,\Phi, \CF^\bullet)$ over $L$ is \textit{weakly admissible} if the inequality 
$$ \sum_i i\dim gr^i_{\CF^\bullet(V)}(N' \otimes_{K_0}L) \leq \ord_p \,\det(\Phi|N')$$
holds for every subisocrystal $N'$ of $(V,\Phi)$ and with equality in case $N'=(V, \Phi)$.\\

Any $1$-PS $\lambda:\BG_m \longrightarrow \bG_L$ defined over $L$ induces a $\mathbb{Z}$-graded $L$-vector space $$V_L=\bigoplus_{i \in \BZ} V_i^{\lambda}$$
for a finite-dimensional algebraic $K$-representation $(V,\rho)$, 
where the grading comes from the weight spaces $V_i^\lambda=\{v \in V_L \mid \rho(\lambda(s))v=s^iv \}$.
Thus, we naturally have a decreasing exhaustive separated $\BZ$-filtration $\CF_\lambda^\bullet(V)$ on $V_L$ given by
$$\CF^i_\lambda(V)=\bigoplus_{j \geq i} V_j^\lambda.$$ Therefore, a tuple $(b, \lambda ) \in \bG(K_0) \times X_*(\bG_L)$ defines a tensor functor 
\begin{align*}
    \Rep_K(\bG) &\longrightarrow \mathrm{FilIso}^L_{F}(\Phi)\\
    (V, \rho) &\mapsto (N_b(V),\CF_\lambda^{\bullet}(V)).
\end{align*}
Such a pair $(b, \lambda)$ is \textit{weakly admissible} if $(N_b(V),\CF_\lambda^{\bullet}(V))$ is weakly admissible for all $(V, \rho) \in \Rep_K(\bG)$.

\subsubsection{Parameterization of weakly admissible filtrations on isocrystals}

In the following, we fix together with $\bG$ an element $b \in \bG(K_0)$ and a conjugacy class $\{\mu\} \subset X_*(\bG)$ over $\ov K$. \\

The conjugacy class $\{\mu\}$ defines the Shimura field $E:=E(\bG, \{\mu\}) \subset \ov K$. It is the fixed subfield of $\ov K$
under the stabilizer $\Gamma_\mu$ of $\{\mu \}$ in $\Gamma_K $ and is a finite extension of $K$. 
As $\bG$ is quasi-split, $\{\mu \}$ contains an element $\mu$ defined over $E$ by \cite[Lemma 1.1.3]{Ko3}. Therefore, the associated flag variety $\sF:=\sF(\bG, \{\mu\})$, defined over $E$, can be identified as$$
\sF=\bG_E/\bP(\mu).$$ 
Let us point out that the $\ov K$-valued points of $\sF$ are given by $$\{\mu\}/\sim$$ 
where $\sim$ is the par-equivalence relation explained in \cite[Section 4.1.2]{CDHN}, 
which identifies the elements of $\{\mu\}$ defining the same filtration on $\Rep_K(\bG)$. Hence, for a field extension $L$ of $E$, 
a point $x \in \sF(L)$ gives rise to a cocharacter $\mu_x \in \{\mu \}$ defined over $L$ up to par-equivalence (cf. \cite[Remark 6.1.6]{DOR} for more details).  \\

Setting $\br E:=E K_0$, we write $\br \sF$ for the adic analytification of $\sF_{\br E}$. According to \cite[Proposition 1.36\,i)]{RZ}, the set $\sF^{\wa}_b:=\sF(\bG, \{\mu \}, b)^{\wa}$ of weakly admissible filtrations with respect to $b$ in $\sF$, i.e. 
$$\sF^{\wa}_b(L)=\{x \in  \sF(L) \mid (b,\mu_x) \text{ weakly admissible} \}$$
for any field extension $L$ of $\br E$, has the structure of a partially proper open subset of $\br \sF$.
The space $\sF^{\wa}_b$ is the \textit{period domain} attached to the triple $(\bG, \{\mu\}, b)$. 
First, we note that $\sF^{\wa}_b$ only depends on $[b] \in B(\bG)$. Secondly, the natural 
action of $\bJ_b(K) \subset \bG(K_0)$ on $\br \sF$ restricts to an action on $\sF_b^{\wa}$ (cf. \cite[1.35 and 1.36\,i)]{RZ}). In the case that $b$ is $s$-decent\footnote{There exists a positive integer $s$ such that $sv_b$ factors through the quotient $\BG_{m,K_0}$  of $\BD_{K_0}$ and the equality $(b\sigma)^s=(s\nu_b)(p)\sigma^s$ holds in $\bG(K_0) \rtimes  \sigma^{\BZ}$. In our case, there exists always such an $s$ (cf. \cite[Remark 9.1.34]{DOR}).}, we can regard  $\sF^{\wa}_b$ as a partially proper open subset defined over $E_s:=E\BQ_{p^s}$ (cf. \cite[Proposition 1.36\,ii)]{RZ}). 

\subsubsection{Existence of weakly admissible filtrations}\label{ss:existence}  

Let $\BD$ be the algebraic pro-torus over $K$ with character group $\BQ$. 
Kottwitz showed in \cite[Section 4]{Ko1} that there exists for $b \in \bG(K_0)$ a unique morphism $$\nu_b: \BD_{K_0} \longrightarrow \bG_{K_0}$$ which, by the Tannakian formalism, induces the tensor functor 
\begin{align*}
\Rep_K(\bG) &\longrightarrow \mathrm{Grad}(\mathrm{Vec}_{K_0}, \BQ) \\
(V, \rho) &\mapsto \bigoplus_{i \in \BQ} V_i 
\end{align*}
from $\Rep_K(\bG)$ to the category of $\BQ$-graded ${K_0}$-vector spaces, where the grading comes from the slope decomposition of $N_b$. That is the \textit{slope homomorphism}. It has the property that $\nu_b=b\sigma(\nu_b)b^{-1}$ 
and thus, we have a well-defined map 
\begin{align*}
    B(\bG) &\longrightarrow [\Hom_{K_0}(\BD_{K_0}, \bG_{K_0})/\Int(\bG(K_0))]^{\sigma=1}, \\
    [b] &\mapsto [v_b]
\end{align*}
the so called \textit{Newton map}. Let $\bB$ be a Borel subgroup in $\bG$ and $\bT$ a maximal torus contained in 
$\bB$. Further, let $X_*(\bT)_{\BQ}^+$ be the set of dominant rational cocharacters of $\bT$ with respect to $\bB$. The chosen $\bB$ induces a partial order $\leq$ on $X_*(\bT)_\BQ$ where $ \lambda' \leq \lambda$ if and only if 
$\lambda-\lambda'=\sum_{\alpha \in \Delta} n_\alpha \alpha^{\vee} $ with all $n_\alpha \in \BQ_{\geq 0}$. According to \cite[Section 2.1/2.2]{RV} (cf. \cite[Remark 3.3]{CDHN}), there is a unique $\mu \in \{\mu\}$ and a  unique representative $\nu_{[b]} \in [\nu_b]$ for $[b] \in B(\bG)$ lying in $X_*(\bT)_{\BQ}^+$.  \\

Rapoport and Viehmann associated in \cite[Definition 2.3]{RV}  the \textit{set of acceptable elements} to a conjugacy class $\{\mu\}$ by setting
$$A(\bG, \{\mu\}):=\{ [b] \in B(\bG) \mid v_{[b]} \leq \overline{\mu} \} $$ 
where $\overline{\mu} := \frac{1}{[\Gamma_K:\Gamma_{\mu}]} \sum_{\tau \in \Gamma_K/\Gamma_\mu} \tau(\mu) \in X_*(\bT)_\BQ^+$. 
We remark that this set is non-empty and finite by \cite[Lemma 2.5]{RV}. Finally, one obtains the following result. 

\begin{theorem}{\cite[Theorem 9.5.10]{DOR}} \label{necessary} The period domain $\sF(G,\{\mu \}, b)^{\wa}$ is non-empty if and only if $[b] \in A(\bG, \{\mu \})$. 
\end{theorem}

\setenumerate[0]{leftmargin=15mm,itemsep=\the\smallskipamount}
\section{Cohomological computations}\label{s:CoPerDom}
\subsection{Setup} \label{s:Setup}
Let $K=\BQ_p$. In order to apply the results of \cite{OSch, OSt} and ideas of \cite{MR}, we consider $\sF^{\wa}:=\sF^{\wa}(\bG, \{\mu \}, 1)$ with 
 $\bG$ a split connected reductive group over $K$ and $\{\mu\} \subset X_*(\bG)$ such that $\bB:=\bP(\mu)$ is a Borel subgroup of $\bG_E$. \\

This entails a lot of simplifications. Notice first that $1 \in [1]$ is $1$-decent as $\nu:=\nu_1$ is trivial. Furthermore, $\bJ:=\bJ_1=\bG$ by \cite[Remark 9.5.9]{DOR} and $E=K$ since the action of $\Gamma_{K}$ is trivial on $\{\mu\}$. Hence, $\sF$ and $\sF^{\wa}$ are defined over $K$. \\

We set $n:=\dim \sF$ and choose a uniformizer $\pi$ of $K$. Further, we fix an $\inva$ on $\bG$ (cf. section \ref{s:rootdatum}).
We choose a split maximal torus $\bT$ of $\bG$ of rank $d$ such that $\mu \in X_*(\bT)_\BQ$. 
Since all Borel subgroups over $K$ of $\bG$ are $\bG(K)$-conjugated (cf. \cite[Theorem 20.9]{Bor}), we can assume that $(\bT,\bB)$ is a Borel pair (cf. section \ref{s:rootdatum}). 
 This gives rise to a set of simple roots (cf. section \ref{s:rootdatum})
    $$\Delta:=\{\alpha_1, \ldots , \alpha_d\} \subset X^*(\bT)_\BQ.$$
After conjugating $\mu$ with an element of $W$, if necessary, we can assume that $\mu$ lies in the positive Weyl chamber with respect to $\bB$, i.e. 
\begin{equation}\label{positiveChamber}
\langle \alpha, \mu \rangle >0
\end{equation}
for all $\alpha \in \Delta$ (here we used that $\bP(\mu)=\bB$ to get $>$).  
Notice that since $\bG$ is split over $K$ we have $\Gamma_\mu=\Gamma_K$, so $\ov\mu=\mu$. By Theorem \ref{necessary}, we assume that $[1] \in A(\bG, \{\mu\})$, i.e. 
\begin{equation}\label{existence}
\mu=\mu-\nu=\sum_{\alpha \in \Delta} n_\alpha \alpha^{\vee}
\end{equation}
with $n_\alpha \in \BQ_{\geq 0}$. 

\begin{lemma}\label{n_a} For $\mu= \sum_{\alpha \in \Delta} n_\alpha \alpha^{\vee}$, we have $n_\alpha \in \BQ_{> 0}.$ 
\end{lemma}

\begin{proof} 
    Let $\beta \in \Delta$ and consider 
    $$0 < \langle  \beta, \mu \rangle = \sum_{\alpha \in \Delta} n_\alpha \langle \beta , \alpha^{\vee} \rangle = 2n_\beta+   \sum_{\alpha \in \Delta\backslash \{\beta\} } n_\alpha \langle \beta, \alpha^{\vee} \rangle  \leq 2n_\beta$$
    where we used that $ \langle \beta ,  \alpha^{\vee}  \rangle \leq 0$ for all simple roots $\alpha \neq \beta$ (cf. Lemma \ref{rootrel}). 
\end{proof} 
For a weight $\lambda \in X^{*}(\bT)$, let $\CL_\lambda$ be the sheaf on $\sF$ with 
\begin{equation}\label{linebundle}
    \CL_\lambda(U)=\Big\{f \in \CO_{\bG}(\pi^{-1}(U)) \mid f(gb)=-\lambda(b)f(g) \text{ for all } g \in \bG(\ov{K}), b \in \bB(\ov{K} )\Big\}
\end{equation}
for $U \subset \sF$ open (cf. \cite[Part I, 5.8]{J}; note the sign in the definition). 
Here,  $\pi:\bG\rightarrow \sF$ is the natural projection. It is a locally free sheaf of rank 1 (cf. \cite[Part II, 4.1]{J}). For example, $\CL_{2\rho}=\omega_{\sF}$.  We fix a dominant $\lambda \in X^*(\bT)^+$ and set $\CE_\lambda :=\CL_\lambda \otimes \omega_{\sF}.$ 
\subsection{Geometrical properties of the complement of $\sF^{\wa}$}\label{s:GeoProp}  Following \cite[Section 3, p. 536]{O1}, each $\tau \in X_*(\bG)_\BQ$ defines a closed subvariety of $\sF$ by setting $Y_\tau:=\{x \in \sF \mid \mu^{\sL}(x, \tau)<0 \}.$ Here $\mu^{\sL}(\, , \,)$ is the slope function (cf. \cite[Definition 2.2]{M}). Then, for $I \subsetneq \Delta$, we set
\begin{equation*}
    Y_I:= \bigcap_{\alpha \not\in I} Y_{\varpi_\alpha}
\end{equation*}
which is again a closed subvariety of $\sF$.

\begin{lemma}\cite[Lemma 3.1]{O1} Let $I \subsetneq \Delta$. The variety $Y_I$ is defined over $K$. The natural action of $\bG(K)$  on $\sF$ restricts to an action of $\bP_I(K)$ on $Y_I$.
\end{lemma}
Let $Y:=\sF^{\ad} \backslash \sF^{\wa}$. For $I \subsetneq \Delta$ and any subset $W \subset \bG/\bP_I(K)$, we set
$$Z_I^W:=\bigcup_{g \in W} gY_I^{\ad}$$ 
which, in view of the previous lemma, is indeed well-defined. 
\begin{lemma}\cite[Lemma 3.2]{O1} The subset $Z_I^W$ is a closed pseudo-adic subspace of $\sF^{\ad}$ for every compact open subset $W \subset \bG/\bP_I(K)$.
\end{lemma}
Then, by \cite[Corollary 2.4]{O1}, we have the following identification
$$Y=\bigcup_{\substack{I \subset \Delta \\ \lvert \Delta \backslash I \rvert =1}} Z_I^{\bG/\bP_I(K)}.$$ 
For an alternative description of the $Y_I$, which will be important hereinafter, we set 
 \begin{equation}\label{omegaI}
    \Omega_I:=\{w \in W \mid (w\mu, \varpi_\alpha)> 0 \text{ for all } \alpha \not\in I  \}
 \end{equation} 
 for $I \subsetneq \Delta$ (cf. \cite[p. 530]{O1}). Reformulating Lemma \ref{equivalentcond}, we get the following statement. 
 \begin{lemma}\label{fweightandpairing} 
Let $I \subsetneq \Delta $. Then, $w \in \Omega_I$ if and only if $\langle \check{\varpi}_\alpha, w\mu\rangle_\der >0$ for all $\alpha \not\in I$. 
 \end{lemma} 

 \begin{definition}\label{defC_Iw}Let $I \subset \Delta$ and $w \in W^I$. 
    The \textit{generalized Schubert cell} in $\sF$ associated to $w$ is 
         $$C_I(w):=\bP_Iw\bB/\bB=\bigcup_{v \in W_I}C(vw).$$ 
    If $I=\emptyset$, we omit the subscript and call it Schubert cell. 
 \end{definition}

First, it turns out that the $Y_I$ are a union of Schubert cells. 

\begin{proposition}\cite[Proposition 4.1]{O1} For $I \subsetneq \Delta$, we have 
    $$Y_I= \bigcup_{w \in \Omega_I} C(w).$$
\end{proposition}

But for our purposes, we need a description in terms of generalized Schubert cells.  

\begin{proposition}\label{genSchCe} For $I \subsetneq \Delta$, we have 
    $$Y_I= \bigcup_{w \in W^I \cap \, \Omega_I} C_I(w).$$
\end{proposition}

\begin{proof}
    We know by \cite[Proposition 11.1.6]{DOR} that $Y_I= \bigcup_{w \in \Omega_I} \bP_Iw\bB/\bB$. 
    For $w' \in \Omega_I$ exist unique $w \in W^I$ and $v \in W_I$ such that $w'=vw$ and $l(w')=l(v)+l(w)$. Hence, we have 
    $$  \bP_Iw'\bB/\bB=\bP_Ivw\bB/\bB=\bP_Iw\bB/\bB.$$ 
    Since $Y_I$ is closed, this implies that $$Y_I= \bigcup_{w \in W^I \cap \, \Omega_I} C_I(w).\qedhere $$ 
\end{proof}
In addition, we make the following observation for the complement of the $Y_I$ in $\sF$. 
\begin{lemma}\label{complement} 
    Let $I \subsetneq \Delta$ and $w_0 \in W$ the longest element. Then, 
    $$ \sF \backslash Y_I= \bigcup_{v \in W \backslash \Omega_I} vw_0C(w_0).$$
\end{lemma}
\begin{proof} 
Let $v \in W$. We first notice that $vw_0C(w_0)=vw_0 \bB w_0v^{-1}v\bB/ \bB$ is the \glqq coordinate neighborhood \grqq \,of $v\bB/\bB$ in $\sF$, which Kempf describes in \cite[Section 3]{K} (cf. \cite[Corollary 3.5]{K}).
 Then, by \cite[Proposition 6.3 a)]{K},
$$ C(v) \subset vw_0C(w_0).$$ 
Hence, 
$$ \sF \backslash Y_I= \bigcup_{v \in W \backslash \Omega_I} C(v) \subset \bigcup_{v \in W \backslash \Omega_I} vw_0C(w_0).$$
For the other inclusion, let $w \in \Omega_I$ and $v \not\in \Omega_I$. Then, we consider 
$$X:=v^{-1}\overline{C(w)} \cap w_0C(w_0).$$  
It is a closed $\bT$-invariant subset of $w_0C(w_0)$. By \cite[Theorem 3.1]{K}, this is in bijection to a closed $\bT$-invariant subset 
$$H \subset \bU_\bB^-.$$
Here $\bT$ acts by conjugation. We suppose that $X$ is non-empty. Thus, $H$ is non-empty. Furthermore, by \cite[Exercise 8.4.6 (5)]{Sp},
$$ \bU_\bB^-(\ov K)=\big\{g \in \bG(\ov K)  \mid \lim_{t \rightarrow 0} (w_0\mu)(t)g(w_0\mu)(t)^{-1}=1  \big\}.$$
As $H$ is closed and $\bT$-invariant, this description implies that $1 \in H$ (cf. \cite[Lemma 9]{Kn}). Therefore, $\bB/\bB \in X$ and 
$v\bB/\bB \in \overline{C(w)}$, respectively. This implies that $v \leq w$ and therefore $v \in \Omega_I$ since $Y_I$ is closed. That is a contradiction. Hence, 
$$ C(w) \cap vw_0C(w_0) = \emptyset$$ 
which implies  $$Y_I \cap \bigcup_{v \in W \backslash \Omega_I} vw_0C(w_0)= \emptyset.\qedhere$$ 
\end{proof}

\subsection{Algebraic local cohomology}\label{s:AlgLocCo}
In this subsection, we consider the local cohomology groups of $\sF$ with support in a (generalized) Schubert cell and in the closed varieties $Y_I$, respectively, with coefficients in $\CE_\lambda$. \\ 

Jantzen states in \cite[Introduction and Part II, Section 1]{J} that split reductive groups and constructions like Borel und Parabolic subgroups can be carried out over $\BZ$, and therefore, by base change, over any integral domain (cf. \cite[Exp. XXV, Corollary 1.3]{SGA3III}).  That means that there is a split connected reductive algebraic group $\sG$ over $\BZ$ with split maximal torus $\sT$ and Borel $\sB$ such that $\sG_K = \bG, \sT_K = \bT$ and $\sB_K = \bB$. Let $\sF_{\BZ}:=\sG/\sB$ and $$C(w)_{\BZ}:=\sB w\sB/\sB \subset \sF_{\BZ}$$ for $w \in W$. They are flat $\BZ$-schemes by \cite[Part I, Section 5.7 (2)]{J}. Moreover, $$\sF=(\sF_{\BZ})_K \text{ and } C(w)=(C(w)_{\BZ})_K.$$ 
The first identity follows from the fact that the base change commutes with the quotient (cf. \cite[Part I, Section 5.5 (4)]{J}). The latter one can be seen after identifying both sides with affine spaces (cf. \cite[Part II, Section 13.3 (1)]{J}).
In \cite[Section 3]{KL} it is mentioned that the flag varieties and Schubert cells admit \glqq flat lifts to 
$\mathbb{Z}$-schemes\grqq. Furthermore, as described in \cite[Section 13, p. 389]{K} (cf. \cite[Part I, Section 5.8]{J}), we have an invertible sheaf $\CE_{\lambda,\BZ}$ on $\sF_\BZ$ which is defined similarly to $\CE_\lambda$. The same arguments apply if we assume that $\bG$ and all introduced objects are defined over $\BQ$ and $\BC$, respectively. We denote the ground field, if it is not $K$, as a subscript in the 
following proof. \\  

Then, we have the following two identifications of local cohomology groups on $\sF$. They are already known over $\BC$ (cf. \cite[(3.3)]{MR} and \cite[Theorem 1 \& Theorem 3]{MR}). 

\begin{lemma}\label{loccocell}For $w \in W$, one has 
    $$H^i_{C(w)}(\sF,\CE_\lambda)\cong\begin{cases} M(w\cdot\lambda) &i=n-l(w),\\ 0 &\text{else} \end{cases}$$
in $\CO_{\alg}$.
\end{lemma}
\begin{proof}
    As $C(w)$ is affine, it follows that $H^i_{C(w)}(\sF,\CE_\lambda)=0$ for $i \neq n-l(w)$ (cf. \cite[Theorem 10.9]{K}). Since $\CE_\lambda$ has a natural $\fkg$-module structure (cf. \cite[Section 1.2]{O2}), we see by functoriality that $H^{n-l(w)}_{C(w)}(\sF,\CE_\lambda)$ is a $\fkg$-module. 
    Furthermore, by \cite[Lemma 12.8.]{K}, we have that $H^{n-l(w)}_{C(w)}(\sF,\CE_\lambda)$ is $\fkt$-semisimple and  $$\ch\big(H^{n-l(w)}_{C(w)}(\sF, \CE_\lambda)\big)=\ch\big(M(w \cdot \lambda)\big).$$ This implies that $H^{n-l(w)}_{C(w)}(\sF,\CE_\lambda)$ lies in the category $\CO_{\alg}$ (cf. \cite[Example 1.1]{AL}). 
    In particular for $w=e$, we see by the last remark in \cite[Section 12]{K} and the proof of \cite[Proposition 1.4.2]{O2} that 
    \begin{align*}
        H^{n}_{C(e)}(\sF, \CE_\lambda)&\cong H^{n}_{C(e)}(\sF, \CO_{\sF}) \otimes_{K} (K)_{2 \rho + \lambda} \
    \cong M(-2\rho)^\vee \otimes_{K} (K)_{2 \rho + \lambda} \\
    &\cong M(-2\rho)\otimes_{K} (K)_{2 \rho + \lambda} \cong M(\lambda)
    \end{align*}
    holds in the category $\CO_{\alg}$. Here, we used \cite[Proposition 7]{B} for the second isomorphism and the fact that $-2\rho$ is antidominant for the third.
    Thus, by Lemma \ref{Verma}, it remains  to prove that 
    there is a non-trivial injective morphism 
    \begin{equation}\label{search}
        H^{n-l(w)}_{C(w)}(\sF,\CE_\lambda) \longrightarrow  H^{n}_{C(e)}(\sF, \CE_\lambda).
    \end{equation}
    For this, let $k \in \{K, \BQ, \BC\}$. Further, we let $X_1:=\overline{C(w)_\BZ}$ and $X_2:=X_1 \backslash C(w)_\BZ$. By Lemma \ref{locrel} and Proposition \ref{excision}, we have  
    $$H^q_{C(w)_\BZ}(\sF_\BZ, \CE_{\lambda,\BZ}  \otimes k ) \cong H^q_{X_1/X_2}(\sF_\BZ, \CE_{\lambda,\BZ}  \otimes k) \text{ and } H^q_{C(w)_k}(\sF_k, \CE_{\lambda,k} ) \cong H^q_{X_{1,k}/X_{2,k}}(\sF_k, \CE_{\lambda,k} ).$$  
    Then, by \cite[Lemma 13.8]{K}, we obtain an isomorphism 
    $$H^q_{C(w)_\BZ}(\sF_\BZ, \CE_{\lambda,\BZ} \otimes k) \cong H^q_{C(w)_k}(\sF_k, \CE_{\lambda,k})$$
    of $k$-vector spaces. As $\CO_{\sF_\BZ}$ is flat over $\BZ$ and $\CE_{\lambda,\BZ}$ is locally free, it follows that $\CE_{\lambda,\BZ}$ is flat over $\BZ$ since it is a local property. 
     Following \cite[Section 4]{KL}, this yields a spectral sequence 
    $$ E_2^{p,q}=\Tor^{\BZ}_{-p}\big(H^q_{C(w)_\BZ}(\sF_\BZ, \CE_{\lambda,\BZ} ), k\big) \Rightarrow H^{p+q}_{C(w)_k}(\sF_k,\CE_{\lambda,k}).$$
    Since $k$ is flat over $\BZ$, we have an isomorphism 
    $$H^{n-l(w)}_{C(w)_\BZ}(\sF_\BZ, \CE_{\lambda,\BZ} )\otimes k \cong H^{{n-l(w)}}_{C(w)_k}(\sF_k,\CE_{\lambda,k})$$
    of $k$-vector spaces. This implies 
    \begin{align}
        H^{{n-l(w)}}_{C(w)_\BC}(\sF_\BC,\CE_{\lambda,\BC}) &\cong  H^{n-l(w)}_{C(w)_\BQ}(\sF_\BQ,\CE_{\lambda,\BQ}) \otimes_\BQ \BC \label{C}
    \intertext{ and }
    H^{{n-l(w)}}_{C(w)}(\sF,\CE_{\lambda}) &\cong  H^{n-l(w)}_{C(w)_\BQ}(\sF_\BQ,\CE_{\lambda,\BQ}) \otimes_\BQ K. \label{K}
    \end{align}
    If we choose $X_1=\sF_\BZ$ and $X_2=\emptyset$ instead, then using the same arguments as before, we get that 
    \begin{align}
        H^q(\sF_\BC,\CE_{\lambda,\BC})&\cong  H^q(\sF_\BQ,\CE_{\lambda,\BQ}) \otimes \BC \label{HC}
    \end{align}
    for all integers $q$. Next, let $Z_j \subset \sF_\BQ$ be the union of the closure of Schubert cells of codimension greater than or equal to $j$. This defines a filtration on $\sF_\BQ$ by closed subsets 
    \begin{equation}\label{globalfilt}
        \sF_\BQ=Z_0 \supset Z_1 \supset \ldots \supset Z_n=C(e)_\BQ.
    \end{equation}
    Furthermore, 
    $$Z_j\backslash Z_{j+1} = \bigsqcup_{\substack{w \in W \\ l(w)=n-j}} C(w)_\BQ.$$ 
    Then, we get, by Lemma \ref{locrel} and Proposition \ref{excision} (cf. \cite[p. 385]{K}), that 
    $$ H^{i}_{Z_j\backslash  Z_{j+1}}(\sF_\BQ,\CE_{\lambda,\BQ})\cong\bigoplus_{\substack{w \in W \\ l(w)=n-j}} H^i_{C(w)}(\sF_\BQ,\CE_{\lambda,\BQ})$$ 
    for all integers $i$. 
    Thus, by Lemma \ref{cousinspectral} and Lemma \ref{loccocell}, we can compute $H^*(\sF_\BQ, \CE_{\lambda,\BQ})$ by the complex 
    \begin{equation}\label{BGGreso}\bigoplus_{\substack{w \in W \\ l(w)=n}} H^0_{C(w)_\BQ}(\sF_\BQ, \CE_{\lambda,\BQ})\rightarrow \ldots \rightarrow  \bigoplus_{\substack{w \in W \\ l(w)=1}} H^{n-1}_{C(w)_\BQ}(\sF_\BQ,\CE_{\lambda,\BQ})\rightarrow H^n_{C(e)_\BQ}(\sF_\BQ, \CE_{\lambda,\BQ}).
    \end{equation} 
    On the other hand, we have by Serre duality (cf. \cite[Part II, 4.2 (9)]{J}, note that the setting in loc. cit. induces different signs) that 
    $$H^i(\sF_\BQ, \CE_{\lambda,\BQ})=H^i(\sF_\BQ, \CL_{\lambda,\BQ} \otimes \omega_{\sF_\BQ}\\) \cong (H^{n-i}(\sF_\BQ, (\CL_{\lambda,\BQ})^\vee))'=(H^{n-i}(\sF_\BQ, \CL_{-\lambda,\BQ}))'.$$
    For the latter one, the Borel-Weil-Bott theorem (cf. \cite[Part II, Corollary 5.5]{J}, note again the setting in loc.cit) gives  $H^{i}(\sF_\BQ, \CL_{-\lambda,\BQ})=0$ for $i \neq 0$. Hence, the complex (\ref{BGGreso}) is a resolution of $H^n(\sF_\BQ, \CE_{\lambda,\BQ})$. By (\ref{C}), (\ref{HC}) and (faithfully) flatness of field extensions, we get an acyclic complex 
    \begin{align}\label{BGGresoC} 0 \rightarrow \bigoplus_{\substack{w \in W \\ l(w)=n}} H^0_{C(w)_\BC}(\sF_\BC, \CE_{\lambda,\BC})\rightarrow \ldots &\rightarrow  \bigoplus_{\substack{w \in W \nonumber \\ l(w)=1}} H^{n-1}_{C(w)_\BC}(\sF_\BC,\CE_{\lambda,\BC}) \\ &\rightarrow H^n_{C(e)_\BC}(\sF_\BC, \CE_{\lambda,\BC})\rightarrow H^n(\sF_\BC, \CE_{\lambda,\BC}) \rightarrow 0.
    \end{align} 
    Again by the Borel-Weil-Bott theorem, we know that $H^n(\sF_\BC, \CE_{\lambda,\BC})= L(\lambda)_\BC$. 
    Here $L(\lambda)_\BC$ is the unique simple quotient of the Verma module $M(\lambda)_\BC$ in the usual BGG Category \CO over the complex numbers (cf. \cite[Section 1.3]{H2}). Then, by \cite[(3.3)]{MR}, we have 
    $$ H^{n-l(w)}_{C(w)_\BC}(\sF_\BC,\CE_{\lambda,\BC}) \cong M(w \cdot \lambda)_\BC$$ 
    for all $w \in W$. Therefore, the complex (\ref{BGGresoC}) is a BGG resolution of $L(\lambda)_\BC$ (cf. \cite[Section 6.1]{H2}).
    Thus, by \cite[Theorem, Section 6.8]{H2}, the natural morphism 
    $$H^{n-l(w)}_{C(w)_\BC}(\sF_\BC,\CE_{\lambda,\BC}) \longrightarrow H^{n-l(w')}_{C(w')_\BC}(\sF_\BC,\CE_{\lambda,\BC})$$ 
    is non-trivial for $w' \leq w$ with $l(w)=l(w)+1$. Moreover, it is injective by \cite[Theorem, Section 4.2]{H2}. This implies that 
    the morphism 
    $$H^{n-l(w)}_{C(w)_\BQ}(\sF_\BQ,\CE_{\lambda,\BQ}) \longrightarrow H^{n-l(w')}_{C(w')_\BQ}(\sF_\BQ,\CE_{\lambda,\BQ})$$ 
    in the complex (\ref{BGGreso}) was already injective by the faithfully flatness of field extensions. Again by the faithfully flatness and by (\ref{K}), we get an injective morphism 
    $$H^{n-l(w)}_{C(w)}(\sF,\CE_{\lambda}) \longhookrightarrow H^{n-l(w')}_{C(w')'}(\sF,\CE_{\lambda})$$ 
    for all $w, w' \in W$ with $w' \leq w$ and $l(w)=l(w')+1$. Let $w \in W$ with reduced expression $w=s_1\ldots s_t$ and let $w_i:=s_1\ldots s_i$, i.e. $w=w_t$. Then, we get the desired morphism (\ref{search}) from the sequence of injections 
    $$ H^{n-l(w)}_{C(w)}(\sF,\CE_{\lambda}) \longhookrightarrow H^{n-l(w_{t-1})}_{C(w_{t-1})}(\sF,\CE_{\lambda}) \longhookrightarrow \ldots \longhookrightarrow H^{n}_{C(e)}(\sF,\CE_{\lambda}).$$ 
    \end{proof}
Similar to Lemma \ref{loccocell}, we have the following identification for generalized Schubert cells. 
\begin{lemma}\label{C_I(w)} For $I \subset \Delta$ and $w \in W^I$, one has
    $$H^i_{C_I(w)}(\sF,\CE_\lambda)\cong\begin{cases} M_I(w\cdot\lambda) &i=n-l(w),\\ 0 &\text{else} \end{cases}$$
in $\CO_{\alg}^{\fkp_I}$.
\end{lemma}

\begin{proof}
Over $\mathbb{C}$ this is \cite[Theorem 1/Theorem 3]{MR}. The arguments of loc. cit. are applicable as well. For completeness, we will recall them. \\

Let $Z_j$ be the union of the closure of Schubert cells of codimension greater than or equal to $j$ and $U_w=\sF\backslash \big(\overline{C_I(w)}\backslash C_I(w)\big)$. Hence, $C_I(w)$ is closed in the open subset $U_w$. Let $r_I:=\dim(\bP_I/\bB)$ and $t:=n-l(w)-r_I$. Then, we consider the filtration on $U_w$ by closed subsets 
\begin{equation}\label{filt}
    U_w \supset C_I(w) \supset C_I(w)  \cap Z_{t+1} \supset \ldots \supset C_I(w)  \cap Z_{r_I+t}=C(w).
\end{equation}
As 
$$ (C_I(w)  \cap Z_{t+j})\backslash (C_I(w)  \cap Z_{t+j+1})= C_I(w)  \cap  (Z_{t+j}\backslash Z_{t+j+1})=\bigsqcup_{\substack{v \in W_I \\ l(v)=r_I-j}} C(vw),$$
we get, as in Lemma \ref{loccocell}, that
$$H^i_{(C_I(w) \cap Z_{t+j})/(C_I(w) \cap Z_{t+j+1})}(U_w,\CE_\lambda)\cong\bigoplus_{\substack{v \in W_I \\ l(v)=r_I-j}} H^i_{C(vw)}(U_w,\CE_\lambda)$$
for all integers $i$. Notice that by Proposition \ref{excision}, we have 
$$H^i_{C(vw)}(U_w,\CE_\lambda)\cong H^i_{C(vw)}(\sF,\CE_\lambda).$$ 
Then, applying Lemma \ref{cousinspectral} to (\ref{filt}) and taking Lemma \ref{loccocell}
into account, we see that the cochain complex 
\begin{equation}\label{complexparabolic}
     \bigoplus_{\substack{v \in W_I \\ l(v)=r_I}} M(v\cdot \lambda) \rightarrow \ldots \rightarrow \bigoplus_{\substack{v \in W_I \\ l(v)=1}} M(v\cdot \lambda) \rightarrow  M(\lambda), 
\end{equation} 
starting in degree $t$, computes $H^*_{C_I(w)}(U_w,\CE_\lambda)$, and therefore, by Proposition \ref{excision}, also 
$H^*_{C_I(w)}(\sF,\CE_\lambda)$. Then, by the work of Lepowsky (cf. \cite[p. 506, Proof of Theorem 4.3]{Le1}), we get 
 $$H^{n-l(w)}_{C_I(w)}(\sF,\CE_\lambda)\cong M_I(w \cdot \lambda).$$  
On the other hand, the complex (\ref{complexparabolic}) is obtained from the BGG-resolution of $V_I(\lambda)$ by Verma modules for $L_{P_I}$ by tensoring with $U(\fkg)$ over $U(\fkp_I)$. This functor preserves exactness.  Therefore, $H^i_{C_I(w)}(\sF,\CE_\lambda)=0$ for $i \neq n-l(w)$. 

\end{proof}
Another application of Lemma \ref{cousinspectral} is the computation of the local cohomology groups $H_{Y_I}^*(\sF,\CE)$.
     \begin{lemma}\label{cohY_I} Let $I \subsetneq \Delta$, $d_I:= \dim(Y_I)$ and $r_I:= \dim(\bP_I/\bB)$. Then, the
        cochain complex

        $$C_I^\bullet:\bigoplus_{\substack{w \in W^I \cap \,\Omega_I \\ l(w)=d_I-r_I}} H^{n-l(w)}_{C_I(w)}(\sF,\CE_\lambda) \rightarrow \bigoplus_{\substack{w \in W^I \cap \,\Omega_I \\ l(w)=d_I-r_I-1}} H^{n-l(w)}_{C_I(w)}(\sF,\CE_\lambda) \rightarrow \ldots \rightarrow  H^{n}_{C_I(e)}(\sF,\CE_\lambda), $$
        with the natural morphisms starting in degree $n+r_I-d_I$, computes the cohomology groups $H^j_{Y_I}(\sF,\CE_\lambda)$. More specifically, $H^j(C_I^\bullet)=H^j_{Y_I}(\sF,\CE_\lambda)$.
    \end{lemma}

    \begin{proof}
    Let $\tilde{Z}_j$ be the closure of the union of those $P_I$-orbits $C_I(w)$ whose codimension is greater or equal to $j$. This defines a filtration 
    $$\sF=\tilde{Z}_0 \supset \tilde{Z}_1 \supset \ldots \supset \tilde{Z}_{n-r}$$by closed subsets. 
    Then, we consider the filtration of closed subsets on $Y_I$ induced by setting $Z_j:=Y_I\cap \tilde{Z}_{n-d_I+j}$ 
    \begin{equation}\label{fil1} 
        Y_I=Z_0\supset  Z_1 \supset \ldots \supset Z_{d_I}
    \end{equation}
    where, by Lemma \ref{genSchCe}, one has $${Z}_j\backslash {Z}_{j+1}=\bigsqcup_{\substack{w \in W^I \cap \,\Omega_I \\ l(w)=d_I-r_I-j}} C_I(w).$$
    Therefore, as in the proof of Lemma \ref{loccocell} , we obtain that 
    $$H^i_{Z_j/Z_{j+1}}(\sF,\CE_\lambda)\cong\bigoplus_{\substack{w \in W^I \cap \,\Omega_I \\ l(w)=d_I-r_I-j}} H^i_{C_I(w)}(\sF,\CE_\lambda).$$ 
    for all integers $i$. Applying Lemma \ref{cousinspectral} to the filtration (\ref{fil1}) and $\CE_\lambda$, and taking Lemma \ref{C_I(w)} into account, we see that the induced spectral sequence 
    $$E_1^{pq}=H^{p+q}_{Z_p/Z_{p+1}}(X, \CE_\lambda)\Rightarrow H^{p+q}_{Y_I}(X,\CE_\lambda)$$ degenerates at the $E_2$-page. Thus, the result follows. 
\end{proof}

\begin{remark}\label{morphismBGG}
    As pointed out in \cite[Theorem 2]{MR}, the morphisms $$H^{n-l(w)}_{C_I(w)}(\sF, \CE_\lambda) \rightarrow H^{n-l(w')}_{C_I(w')}(\sF, \CE_\lambda)$$ for $w' \leq w$ with $l(w)=l(w')+1$ that appear in the differentials are those
    from the Lepowsky BGG resolution (cf. \cite[Theorem 4.3]{Le1}). 
\end{remark}

\begin{corollary}\label{loccohcatOp}
        For each $i \in \BN_0$, the $U(\fkg)$-module $H_{Y_I}^i(\sF,\CE_\lambda)$  lies in $\CO_{\alg}^{\fkp_I}$.
\end{corollary}
    
\begin{proof}
        As the category $\CO_{\alg}^{\fkp_I}$ is closed under taking submodules and quotients, this follows immediately from Lemma \ref{C_I(w)} and \ref{cohY_I}.
\end{proof}

\subsection{Analytic local cohomology}\label{s:AnaLocCo} In the following, we would like to relate the algebraic local cohomology groups of $\sF$ with support in the $Y_I$ and coefficients in $\CE_\lambda$ to some analytic local cohomology groups of $\sF^{\rig}.$ \\

For this, we recall first from \cite[Section 1.3]{O2} what we mean by analytic local cohomology. Let $X$ be a rigid analytic variety over $K$ and $U \subset X$ an admissible open subset with $Z:=X\backslash U$, the set theoretical complement. Further, be $\CE$ be a coherent sheaf on $X$. 
Then, similar to the last section, we define 
$$ \Gamma_Z(X,\CE):=\ker\big(\Gamma(X,\CE)\rightarrow \Gamma(U,\CE)\big) $$ 
and $H^*_Y(X,\CG)$ to be the right derived functors. In case $X$ is a separated rigid analytic variety of countable type, the local cohomology groups carry a natural structure of a locally convex $K$-vector space which is in general not Hausdorff (cf. \cite[Section 1.3]{O2}, \cite[Section 1.6]{vP}). \\

For our purposes, we fix an embedding $\sF \hookrightarrow{} \BP_{K}^N$ defined by the vanishing ideal $\sI \subset \CO_K[T_0,\ldots, T_N]$. We introduce, adapted from \cite[Section 2, p. 1398]{O6}, the notion of special neighborhoods of 
a closed subvariety of $\sF$. They play a crucial role in the computation of the cohomology of a period domain. 
\begin{definition}\label{epsnbh}
Let $\epsilon \in \lb \overline{K}^\times \rb$. Let $Y \subset \sF$ be a closed subvariety 
and $f_1, \ldots, f_r \in \CO_K[T_0,\ldots, T_N]$ homogeneous polynomials such that they generate the vanishing ideal of the Zariski closure of $Y$ in $\sF_{\CO_K}$. Additionally, each $f_i$ has at least one coefficient in $\CO_K^\times$. 
\begin{enumerate}[label=\roman*)]
    \item  We call a tuple $(z_0,\ldots, z_{N}) \in \BA_K^{N+1}(C)$ \textsl{unimodular} if $z_i \in \CO_{C}$ 
           for all $i$ and there exists an $i$ such that $z_i \in \CO_{C}^\times$.
    \item  We define the \textsl{open $\epsilon$-neighborhood} of $Y$ in $\sF^{\rig}$ by 
\begin{align*}
    Y(\epsilon):=\Big\{z \in \sF^{\rig} \, \Big\vert \, &\text{for any unimodular representative $\tilde{z}$ of $z$, we have } \\ &\lb f_j(\tilde{z})\rb \leq \epsilon \text{ for all $j$}\Big\}.
\end{align*}
\item We define the \textsl{closed $\epsilon$-neighborhood} of $Y$ in $\sF^{\rig}$ by 
\begin{align*}
    Y^{-}(\epsilon):=\Big\{z \in \sF^{\rig} \, \Big\vert \, &\text{for any unimodular representative $\tilde{z}$ of $z$, we have } \\ &\lb f_j(\tilde{z})\rb < \epsilon \text{ for all $j$}\Big\}.
\end{align*}
\end{enumerate}
\end{definition}
Let $I \subsetneq \Delta$. Let $\Phi_{\fku_{\bP_I}^-}=\{{\alpha_1}, \ldots, {\alpha_r}\}$ be the set of roots appearing in $\fku_{\bP_I}^-$ (under the adjoint action of $\bT$) and $y_{\alpha_1}, \ldots, y_{\alpha_r}$ be a basis of the $K$-vector space $\fku_{\bP_I}^-$. Then, for $\epsilon \in \lb \overline{K^*} \rb$, the norm ${\lb \enspace\, \rb}_\epsilon$ on $U(\fku_{\bP_I}^-)$ is given by 
\begin{equation}\label{normUp}
   \Bigg \lb \sum_{(i_1,\ldots,i_r) \in \BN_0^r} a_{i_1,\ldots,i_r}y_{\alpha_1}^{i_1}\cdots  y_{\alpha_r}^{i_r} \Bigg\rb_\epsilon= \sup_{(i_1,\ldots,i_r) \in \BN_0^r}\Big\lb i_1! \cdots i_r! \cdot a_{i_1,\ldots,i_r}\Big\rb \epsilon^{i_1 + \ldots +i_r}.
   \end{equation}
Completing $U(\fku_{\bP_I}^-)$ with respect to ${\lb \enspace\, \rb}_\epsilon$ yields the $K$-Banach space
 \begin{align}\label{completeUEA}
   U(\fku_{\bP}^-)_\epsilon:=\Bigg\{&\sum_{(i_1,\ldots,i_r) \in \BN_0^r} a_{i_1,\ldots,i_r}y_{\alpha_1}^{i_1}\cdots  y_{\alpha_r}^{i_r}\, \bigg\vert \, a_{i_1,\ldots,i_r} \in K, \nonumber \\
   &\lb i_1! \cdots i_r! \cdot a_{i_1,\ldots,i_r}\rb \epsilon^{i_1 + \ldots +i_r} \rightarrow 0 \text{ for } i_1+\ldots + i_r \rightarrow 0
   \Bigg\}.
\end{align}
Let $m\in \BN$ and $\epsilon_m:=\lb \pi \rb^m$. We will write $U(\fku_{\bP_I}^-)_m$ for $U(\fku_{\bP_I}^-)_{\frac{1}{\epsilon_m}}$. Let $i \in \BN_0$.  We know from Corollary \ref{loccohcatOp} that $H^i_{Y_I}(\sF,\CE_\lambda) \in \mathcal{O}_{\alg}^{\fkp_I}$. Thus, we have an $\CO^{\fkp_I}_{\alg}$-pair $(H^i_{Y_I}(\sF,\CE_\lambda),W)$ with a short exact sequence 
\begin{equation}\label{seqCoh}
      0 \rightarrow \fkd \rightarrow U(\fku_{\bP_I}^-) \otimes_{K} W \longrightarrow H^i_{Y_I}(\sF,\CE_\lambda) \rightarrow 0.
\end{equation}
Then, the above defines a norm on $U(\fku_{\bP_I}^-) \otimes_{K} W$ and induces, by (\ref{seqCoh}), the quotient norm on $H^i_{Y_I}(\sF,\CE_\lambda)$. 
Since the quotient map is open, it is strict (cf. \cite[Section 1.1.9, Proposition 3 ii)]{BGR}) and we obtain, by \cite[Section 1.1.9, Corollary 6]{BGR}, a short exact sequence of $K$-Banach spaces
\begin{equation}\label{seqComplete}
    0 \longrightarrow \fkd_{m} \longrightarrow U(\fku_{\bP_I}^-)_m \otimes_K W \longrightarrow \hat{H}^i_{Y_I,m} \longrightarrow 0
\end{equation}
with $\fkd_{m}$ the completion of $\fkd$ in $U(\fku_{\bP_I}^-)_m$ and $\hat{H}^i_{Y_I,m}$ denotes the completion of $H^i_{Y_I}(\sF,\CE_\lambda)$ with respect to the quotient norm. \\

Moreover, we define $D_w:=ww_0C(w_0)$ and $H_w:=\sF \backslash D_w$ for $w \in W$. By Lemma \ref{complement}, we know that 
\begin{equation}\label{affinecovering}
    \sF \backslash Y_I=\bigcup_{w \in W \backslash \Omega_I} D_w 
\end{equation}
which is an affine open covering of the complement because $C(w_0)$ is affine open. 
Thus, we can compute $H^*(\sF \backslash Y_I,\CE_\lambda)$ by the \v{C}ech-complex 
$$ \bigoplus_{w \in W \backslash \Omega_I} \Gamma(D_{w}, \CE_\lambda) \rightarrow \bigoplus_{\substack{w, w' \in W \backslash \Omega_I \\ w \neq w'}} \Gamma(D_{w} \cap D_{w'} , \CE_\lambda)  \rightarrow \ldots \rightarrow \Gamma(\bigcap_{w \in W \backslash \Omega_I} D_{w} , \CE_\lambda).$$ Furthermore, we can easily deduce from (\ref{affinecovering}) that for $\epsilon \in \lb \overline{K^\times} \rb$, we have 
$$ Y_I^-(\epsilon)= \bigcap_{w \in W \backslash \Omega_I} H_w^-(\epsilon).$$ 
Therefore, we consider the subset 
$$ D_{w,\epsilon}:=\sF^{\rig} \backslash H_w^-(\epsilon)$$
for $w \in W$. 

\begin{lemma} Let $w \in W$. The subset $D_{w,\epsilon}$ is affinoid. 
\end{lemma}

\begin{proof}
    Since $H_w$ is of codimension 1, there is $f \in \CO_K[T_0,\ldots, T_N]$ homogenous of degree $t$ with at least one coefficient in $\CO_K^{\times}$ and generating the vanishing ideal of the Zariski closure of $H_w$ in $\sF_{\CO_K}$. Let $N_0:= \binom{n+t}{t}-1$. We embed $ \BP^N_K$ into $\BP^{N_0}_K$ via the $t$-th Veronese embedding. Then, by substituting monomials, $f$ yields a homogeneous linear polynomial $g \in \CO_K[T_0,\ldots, T_{N_0}]$ defining a hyperplane $H_0 \subset \BP^{N_0}_K$, such that $$H_w=\sF \cap H_0.$$ It is known that $(\BP^{N_0}_K)^{\rig}\backslash H_0^-(\epsilon)$ is affinoid (cf. \cite[Section 1, Proof of Proposition 4]{SS}) and we notice that 
    $$ \sF^{\rig}\backslash \Big(\sF^{\rig} \cap H_0^-(\epsilon)\Big) = \sF^{\rig}\cap \Big((\BP^{N_0}_K)^{\rig}\backslash H_0^-(\epsilon)\Big).$$ Thus, $\sF^{\rig}\backslash \Big(\sF^{\rig} \cap H_0^-(\epsilon)\Big)$
    is also affinoid since it is a zero set in $(\BP^{N_0}_K)^{\rig}\backslash H_0^-(\epsilon)$. However, we have $\sF^{\rig} \cap H_0^-(\epsilon)=H_w^-(\epsilon)$. 
\end{proof}
This results in the affinoid covering 
$$\sF^{\rig} \backslash Y_I^-(\epsilon)= \bigcup_{w \in W \backslash \Omega_I} D_{w,\epsilon}$$
which has two consequences. On the one hand we get an admissible covering 
$$\sF^{\rig} \backslash Y_I(\epsilon_m)=\bigcup_{{\substack{\epsilon \rightarrow \epsilon_m \\ \epsilon_m < \epsilon \in \lb \overline{K^\times}\rb}}} \sF^{\rig} \backslash Y_I^-(\epsilon) $$
by quasi-compact admissible open subsets (cf. \cite[Section 1.3, p. 601]{O2}).
On the other hand we can compute $H^*(\sF^{\rig} \backslash Y_I^-(\epsilon),\CE_\lambda)$ by the \v{C}ech-complex 
$$ \bigoplus_{w \in W \backslash \Omega_I} \Gamma(D_{w,\epsilon}, \CE_\lambda) \rightarrow \bigoplus_{\substack{w, w' \in W \backslash \Omega_I \\ w \neq w'}} \Gamma(D_{w,\epsilon} \cap D_{w',\epsilon} , \CE_\lambda)  \rightarrow \ldots \rightarrow \Gamma(\bigcap_{w \in W \backslash \Omega_I} D_{w,\epsilon} , \CE_\lambda)
$$ which terms are $K$-Banach spaces. So far, we have to make the following assumption and conjecture (cf. \cite[Proposition 2.5.2]{Li} and \cite[Lemma 3.4.10]{Li} for the Drinfeld case). 
\begin{hac}\label{hypo1}
   Let $I \subsetneq \Delta$ and $i \in \BN_0$. Both, the cohomology groups $H^i(\sF^{\rig} \backslash Y_I^-(\epsilon),\CE_\lambda)$, and $H^i_{Y_I^-(\epsilon)}(\sF^{\rig},\CE_\lambda)$ are
   $K$-Banach spaces in which the algebraic cohomology group $H^i(\sF \backslash Y_I,\CE_\lambda)$ and $H^i_{Y_I}(\sF,\CE_\lambda)$, respectively, is a dense subspace. 
   Moreover, we have an isomorphism of topological $K$-vector spaces 
   $$ \varprojlim_{m \in \BN} H^i_{Y_I^-(\epsilon_m)}(\sF^{\rig},\CE_\lambda) \cong \varprojlim_{m \in \BN} \hat{H}^i_{Y_I,m}.$$ 
\end{hac}

\begin{corollary} \label{limitloc} Let $I \subsetneq \Delta$. For $i \in \BN_0$, we have the following isomorphisms of topological $K$-vector spaces:
    \begin{align*}
        H^i(\sF^{\rig} \backslash Y_I(\epsilon_m),\CE_\lambda)&\cong \varprojlim_{\substack{\epsilon \rightarrow \epsilon_m \\ \epsilon_m < \epsilon \in \lb\overline{K^\times}\rb}} H^i(\sF^{\rig} \backslash Y_I^-(\epsilon),\CE_\lambda) \\
       \intertext{and} 
        H^i_{Y_I(\epsilon_m)}(\sF^{\rig},\CE_\lambda)&\cong \varprojlim_{\substack{\epsilon \rightarrow \epsilon_m \\ \epsilon_m < \epsilon \in \lb \overline{K^\times}\rb}} H^i_{Y_I^{-}(\epsilon)}(\sF^{\rig},\CE_\lambda).
    \end{align*}
\end{corollary}

\begin{proof}
    The proof is same as that of \cite[Lemma 1.3.2]{O2}. 
\end{proof}


Let $\bG_0$ be a split reductive group model of $\bG$ over $\CO_K$ with Borel pair $(\bT_0, \bB_0)$ and a standard parabolic subgroup $\bP_{I,0}$ containing $\bB_0$ such that the base change to $K$ yields the pair $(\bT,\bB)$ and $\bP_I$, respectively. 
For any positive number $m \in \BN$ let 
\begin{equation*}
p_m:\bG_0(\CO_K) \rightarrow \bG_0(\CO_K/(\pi^m))
\end{equation*}
be the natural reduction map. Then, we set $G_0:=\bG_0(\CO_K)$ and define $$P_I^m:=p_m^{-1}(\bP_{I,0}(\CO_K/(\pi^m)) \subset G_0.$$ Notice that $G_0$ is compact. Moreover, let $\sF_{\mathcal{O}_K}:=\bG_0/\bB_0$.
 
\begin{lemma}Let $I \subsetneq \Delta$ and $m \in \BN$. Then, the subset $Y_I(\epsilon_m)$ is $P_I^m$- invariant. 
\end{lemma}
\begin{proof}
    We identify $\sF^{\rig}$ with the closed points of $\sF$, i.e. for $x \in \sF^{\rig}$ exists a finite extension $L:=k(x)$ of $K$ such that $x \in \sF(L)$. Denote by $\lb \text{\,\,\,} \rb_L$ be the unique absolute value on $L$ which extends the valuation on $K$. 
    Since $\sF_{\mathcal{O}_K}$ is proper over $\mathcal{O}_K$, we have $\sF_{\mathcal{O}_K}(O_L)=\sF_{\mathcal{O}_K}(L)=\sF(L)$. Let 
    $$ q_{m,L}:\sF_{\mathcal{O}_K}(\mathcal{O}_L) \rightarrow \sF_{\mathcal{O}_K}(\mathcal{O}_L/\pi^m\mathcal{O}_L)$$ 
    be the natural projection. The (free) action of $\bG_0$ on $\sF_{\mathcal{O}_K}$ induces the commutative diagram 
    $$\begin{CD}
        \bG_0(\mathcal{O}_L)\times \sF_{\mathcal{O}_K}(\mathcal{O}_L) @> \mathrm{mult.} >>  \sF_{\mathcal{O}_K}(\mathcal{O}_L)\\
        @V (p_{m,L},\, q_{m,L}) V V @VV q_{m,L} V\\
        \bG_0(\mathcal{O}_L/\pi^m\mathcal{O}_L)\times \sF_{\mathcal{O}_K}(\mathcal{O}_L/\pi^m\mathcal{O}_L) @>\mathrm{mult.} >>  \sF_{\mathcal{O}_K}(\mathcal{O}_L/\pi^m\mathcal{O}_L).
    \end{CD}$$
    Moreover, it is clear that $Y_{I,0}:=\bigcup_{w \in \Omega_I} \bB_0w\bB_0/\bB_0$ is the Zariski closure of $Y_I$ in $\sF_{\mathcal{O}_K}$ defined by homogenous polynomial $f_1, \ldots, f_r \in \CO_K[T_0,\ldots, T_N]$ as in Definition \ref{epsnbh}. Then, $Y_{I,0}$ is $\bP_{I,0}$ invariant (cf. Proposition \ref{genSchCe}). Therefore, the above diagram implies that \begin{equation} P^m_I \cdot q_{m,L}^{-1}\big(Y_{I,0}(\mathcal{O}_L/\pi^m\mathcal{O}_L)\big)=q_{m,L}^{-1}\big(Y_{I,0}(\mathcal{O}_L/\pi^m\mathcal{O}_L)\big). \end{equation}
    Next we will show that
    $$ q_{m,L}^{-1}\big(Y_{I,0}(\mathcal{O}_L/\pi^m\mathcal{O}_L)\big)=\{y \in \sF_{\mathcal{O}_K}(\mathcal{O}_L) \mid \lvert f_i(y)\lvert_L\leq \epsilon_m \text{ for all }i\}=Y_I(\epsilon_m)(L).$$
    Let $y \in \sF_{\mathcal{O}_K}(\mathcal{O}_L)$. If $y \in q_{m,L}^{-1}\big(Y_{I,0}(\mathcal{O}_L/\pi^m\mathcal{O}_L)\big)$, it follows that $f_i(q_{m,L}(y))=0$ for all $i$. But $$f_i(q_{m,L}(y))=[f_i(y)]\in \mathcal{O}_L/\pi^m\mathcal{O}_L.$$ Hence, $f_i(y) \in \pi^m\mathcal{O}_L$ and $\lb f_i(y) \rb_L\leq \epsilon_m$ for all $i$. If, on the other hand, $\lb f_i(y)\rb_L\leq \epsilon_m$ for all  $i$, we deduce that $f_i(y) \in \pi^m\mathcal{O}_L$ for all $i$. Thus, $f_i(q_{m,L}(y))=[f_i(y)]=0$ for all $i$ and $y \in q_{m,L}^{-1}\big(Y_{I,0}(\mathcal{O}_L/\pi^m\mathcal{O}_L)\big)$. \\

    \noindent From that we conclude that $Y_I(\epsilon_m)(L)$ is $P_I^m$-invariant for all finite extensions $L$ of $K$. This implies that $Y_I(\epsilon_m)$ is $P_I^m$-invariant. 
\end{proof}
The previous lemma yields a natural $P^m_I$-module structure on $H^i_{Y_I(\epsilon_m)}(\sF^{\rig},\CE_\lambda)$ which we fix.

\begin{lemma} \label{Bi-Functor} For $I \subsetneq \Delta$ and $i \in \BN_0$, we have  
\begin{equation*} 
    \Big(\varprojlim_{m \in \BN} \Ind_{P_I^m}^{G_0}\big(H^i_{Y_I(\epsilon_m)}(\sF^{\rig},\CE_\lambda)\big)\Big)'= \CF_{P_I}^G\big(H^i_{Y_I}(\sF,\CE_\lambda)\big).
    \end{equation*}
\end{lemma}

\begin{proof}
    We know from (\ref{seqComplete}), that 
    \begin{equation*}
       \varprojlim_{m \in \BN} \big(U(\fku_{\bP_I}^-)_m \otimes_{K} W/\fkd_{m}\big) \cong  \varprojlim_{m \in \BN} \hat{H}^i_{Y_I,m}.
    \end{equation*}
    Furthermore, by Assumption \ref{hypo1} and in view of Corollary \ref{limitloc}, we have 
    $$ \varprojlim_{m \in \BN} H^i_{Y_I(\epsilon_m)}(\sF^{\rig},\CE_\lambda) \cong \varprojlim_{m \in \BN} \big(U(\fku_{\bP_I}^-)_m \otimes_{K} W/\fkd_{m}\big) $$ 
    compatible with the action of $\varprojlim_{m \in \BN} P^m_I=P_{I,0}$ (cf. \cite[Proposition 1.3.10 + Proof]{O2}).  
    Then, we get (cf. \cite[p. 633]{O2})
    \begin{equation*}
        \varprojlim_{m \in \BN} \Ind_{P_I^m}^{G_0}\big(H^i_{Y_I(\epsilon_m)}(\sF^{\rig},\CE_\lambda)\big)\cong \varprojlim_{m \in \BN} \Ind_{P_I^m}^{G_0}\big(U(\fku_{\bP_I}^-)_m \otimes_{K} W/\fkd_{m}\big).
    \end{equation*}
    Passing to the dual, which is exact on $K$-Fréchet spaces (cf. \cite[Section I, Corollary 1.4]{BS}), the required statement follows from \cite[Corollary 3.12]{OSt}. .

\end{proof}

\subsection{Results}\label{s:results}
 We start by recalling Orlik's fundamental complex on $Y_{\acute{e}t}$, the étale site on $Y$. This is taken from \cite[Section 6.2.1/6.2.2]{CDHN} which is based on \cite[Section 3]{O1}. \\

 For the constant étale sheaf $\BZ \in \Sh(Y_{\acute{e}t})$ and a closed pseudo-adic subspace $Z$ of $Y$ with inclusion $i:Z \rightarrow Y$, define $\BZ_Z:=i_*i^*(\BZ)$.

 \begin{definition}\cite[Definition 6.7]{CDHN} Let $I \subsetneq \Delta$. Define $\BZ_I \in \Sh(Y_{{\acute{e}t}})$ as the \textit{subsheaf of locally constant sections}
of $\prod_{g \in \bG/\bP_I(K)} \BZ_{gY_I^{\ad}}$, i.e. 
$$ \BZ_I= \varinjlim_{c \in \sC_I} \BZ_c$$
the limit being taken over the (pseudo-filtered) category $\sC_I$ of compact open disjoint coverings of $\bG/\bP_I(K)$ ordered by refinement where $\BZ_c$ denotes the image of the natural embedding $\bigoplus_{j \in A} \BZ_{Z_I^{T_j}} \hookrightarrow \prod_{g \in \bG/\bP_I(K)} \BZ_{gY_I^{\ad}}$ for $c=\{T_j\}_{j \in A} \in \sC_I$. 
 \end{definition}

Let $ I \subset I' \subsetneq \Delta$ and $\pi_{I,I'}:\bG/\bP_I(K)\rightarrow \bG/\bP_{I'}(K)$ be the natural surjection. Then, for all $g \in \bG/\bP_I(K)$ and $h \in \bG/\bP_{I'}(K)$, we have a natural morphism 
$\BZ_{gY_I^{\ad}} \rightarrow \BZ_{hY_{I'}^{\ad}}$ which is trivial if $\pi_{I,I'}(g) \neq h$ and otherwise, it coincides with the map induced by the closed embedding $gY_I \hookrightarrow hY_{I'}$.
Then, by definition, we get a natural morphism $$p_{I,I'}:\BZ_{I'}\rightarrow \BZ_I.$$ Fix an ordering on $\Delta$. Assuming that$\lb I' \rb- \lb I \rb =1 $ and $I'=\{\alpha_1 < \ldots < \alpha_r\}$ we set 
$$d_{I,I'}:= \left\{\begin{array}{lll} (-1)^i p_{I,I'} &\text{ if } &I'=I \cup \{\alpha_i\}, \\
                            0 &\text{ else.}  \end{array}\right.$$
This defines by standard procedure the following complex
\begin{equation}\label{fundamentalcomplex}
    0 \longrightarrow \BZ \longrightarrow \bigoplus\limits_{\substack{I \subset \Delta \\ \lb \Delta \backslash I \rb = 1 }} \BZ_I \longrightarrow \bigoplus\limits_{\substack{I \subset \Delta \\ \lb \Delta \backslash I \rb = 2 }} \BZ_I \longrightarrow \ldots \longrightarrow \bigoplus\limits_{\substack{I \subset \Delta \\ \lb \Delta \backslash I \rb = \lb \Delta \rb -1 }} \BZ_I \longrightarrow \BZ_\emptyset \longrightarrow 0
\end{equation}
on $Y_{\acute{e}t}$ 
which is acyclic by  \cite[Theorem 6.9]{CDHN}. It is referred to as the \textit{fundamental complex}. \\

Denote by $\iota:Y \hookrightarrow \sF^{\mathrm{ad}}$ the closed embedding. Then, by \cite[Exp. I, Proposition 2.3]{SGA2}, we have 
\begin{equation*}
    \Ext^*(\iota_*(\BZ_{Y}),\CE_\lambda) \cong H^*_{Y}(\sF^{\mathrm{ad}},\CE_\lambda). 
\end{equation*}
By applying $\Ext^*(\iota_*(-),\CE_\lambda)$ to the complex (\ref{fundamentalcomplex}), we get the spectral sequence 
\begin{equation}\label{spectral}
\hat{E}_1^{-p,q}=\Ext^q(\bigoplus\limits_{\substack{I \subsetneq \Delta \\ \lb \Delta \backslash I \rb = p+1 }} \iota_*(\BZ_I), \CE_\lambda) \Rightarrow \Ext^{-p+q}(\iota_*(\BZ_{Y}),\CE_\lambda)=H^{-p+q}_{Y}(\sF^{\mathrm{ad}},\CE_\lambda). 
\end{equation}
For the $\hat{E}_1$-terms, we have the following identification. 

\begin{proposition}\label{extgroups} For all $I \subsetneq \Delta$, there exists an isomorphism 
    \begin{equation*}
        \Ext^*(\iota_*(\BZ_I), \CE_\lambda) \cong \varprojlim_{m \in \BN} \Ind_{P_I^m}^{G_0} \big( H^*_{Y_I(\epsilon_m)}(\sF^{\rig},\CE_\lambda)\big).
    \end{equation*}
\end{proposition}

\begin{proof} This is essentially the proof of \cite[Proposition 2.2.1]{O2}, where the Drinfeld case is treated. 
    The family $$\{gP^m_I \mid g \in G_0,\, m \in \BN \}$$ 
    of compact open subsets in $G_0/P_I$ yields $$\BZ_I = \varinjlim_{m\in \BN} \bigoplus_{g \in G_0/P_I^m}\BZ_{Z^{gP^m_I}}.$$ 
    Then, by choosing an injective resolution $\CI^\bullet$ of $\CE_\lambda$, we have 
    \begin{align*}
        \Ext^i( \iota_*(\BZ_I), \CE_\lambda)&=H^i\big(\Hom(\iota_*(\BZ_I),\CI^\bullet)\big)= H^i\big(\Hom(\iota_*(\varinjlim_{m\in \BN} \bigoplus_{g \in G_0/P_I^m}\BZ_{Z_I^{gP^m_I}}),\CI^\bullet)\big) \\
                                &=H^i\big(\varprojlim_{m\in \BN} \bigoplus_{g \in G_0/P_I^m} \Hom(\iota_*(\BZ_{Z_I^{gP^m_I}}), \CI^\bullet)\big)\\&=H^i\big(\varprojlim_{m\in \BN}\bigoplus_{g \in G_0/P_I^m} H^0_{Z_I^{gP^m_I}}(\sF^{\ad}, \CI^\bullet)\big).
    \end{align*}
    We set $Z_{I,m}:=P^m_I\cdot Y_I^{\rig} \subset \sF^{\rig}$ for $m \in \BN$. Then, we have chains of open admissible subsets 
    $$\ldots \sF^{\rig} \backslash Z_{I,m} \subset \sF^{\rig} \backslash Z_{I,{m+1}} \subset \ldots $$
    and  
    $$ \ldots \sF^{\rig} \backslash  Y_I(\epsilon_m) \subset \sF^{\rig} \backslash Y_I(\epsilon_{m+1}) \subset  \ldots $$ 
    which each cover $\sF^{\rig}\backslash Y_I^{\rig}$. Then, we know from the proof of \cite[Section 2, Proposition 4]{SS} that 
    $$ \varprojlim_{m\in \BN} H^0_{Z_{I,m}}(\sF^{\rig}, \CI^p)= H^0_{Y_I^{\rig}}(\sF^{\rig}, \CI^p)= \varprojlim_{m\in \BN} H^0_{Y_I(\epsilon_m)}(\sF^{\rig}, \CI^p).$$  The same holds for translates of  $Z_{I,m}$ and $Y_I(\epsilon_m)$. Hence, we get (cf. \cite[p. 1415]{O6}) 
    $$  \varprojlim_{m\in \BN} \bigoplus_{g \in G_0/P_I^m} H^0_{gZ_{I,m}}(\sF^{\rig}, \CI^p) \cong  \varprojlim_{m\in \BN} \bigoplus_{g \in G_0/P_I^m} H^0_{gY_I(\epsilon_m)}(\sF^{\rig}, \CI^p) $$ 
    for all injective sheafs of the resolution $\CI^\bullet$. Therefore, using that $\sF^{\rig}$ and $\sF^{\ad}$ have equivalent topoi (cf. \cite[Proposition 2.1.4]{Hu}), we get by functoriality an isomorphism of complexes 
    $$ \varprojlim_{m\in \BN}\bigoplus_{g \in G_0/P_I^m} H^0_{Z_I^{gP^m_I}}(\sF^{\ad}, \CI^\bullet) \cong \varprojlim_{m\in \BN}\bigoplus_{g \in G_0/P_I^m} H^0_{gY_I(\epsilon_m)}(\sF^{\rig}, \CI^\bullet).$$ 
    This implies 
    $$ \Ext^i( \iota_*(\BZ_I), \CE_\lambda)= H^i\big(\varprojlim_{m\in \BN}\bigoplus_{g \in G_0/P_I^m} H^0_{gY_I(\epsilon_m)}(\sF^{\rig}, \CI^\bullet)\big).$$

Before we can continue, we need a technical lemma where $\varprojlim_{m \in \BN}^{(r)}$ denote the $r$-th right derived functor of ${\varprojlim_{m \in \BN}}$.
\begin{lemma} Let $\CI$ be an injective sheaf on $\sF^\ad$. Then,
    $$ {\varprojlim_{m \in \BN}}^{(r)}\big(\bigoplus_{g \in G_0/P_I^m} H^0_{gY_I(\epsilon_m)}(\sF^{\rig}, \CI)\big)=0 \text{ for all } r \geq 1.$$
\end{lemma}
\begin{proof}
    It is sufficient to reproduce the proof of  \cite[Lemma 2.2.2]{O2}.
\end{proof}

Then, with the two standard hypercohomology spectral sequences 
\begin{align*}
    E_1^{pq}&={\varprojlim_{m \in \BN}}^{(q)}\big(\bigoplus_{g \in G_0/P_I^m} H^0_{gY_I(\epsilon_m)}(\sF^{\rig}, \CI^p)\big)\Rightarrow \bH^{p+q}{\varprojlim_{m \in \BN}} \big(\bigoplus_{g \in G_0/P_I^m} H^0_{gY_I(\epsilon_m)}(\sF^{\rig}, \CI^\bullet)\big), \\
    E_2^{pq}&={\varprojlim_{m \in \BN}}^{(p)}H^q\big(\bigoplus_{g \in G_0/P_I^m} H^0_{gY_I(\epsilon_m)}(\sF^{\rig}, \CI^\bullet)\big)\Rightarrow \bH^{p+q}{\varprojlim_{m \in \BN}} \big(\bigoplus_{g \in G_0/P_I^m} H^0_{gY_I(\epsilon_m)}(\sF^{\rig}, \CI^\bullet)\big), 
    \end{align*}
and by knowing that $\varprojlim_{m \in \BN}^{(p)}=0$ for $p\geq 2$ (cf. \cite{Je}), we get the following short exact sequence (cf. \cite[Section 2, Proof of Proposition 4]{SS})
\begin{align*} 0 \rightarrow {\varprojlim_{m \in \BN}}^{(1)}\bigoplus_{g \in G_0/P_I^m} H^{i-1}_{gY_I(\epsilon_m)}(\sF^{\rig}, \CE_\lambda) &\rightarrow \Ext^i( i_*(\BZ_I), \CE_\lambda) \\ &\rightarrow {\varprojlim_{m \in \BN}}\bigoplus_{g \in G_0/P_I^m} H^{i}_{gY_I(\epsilon_m)}(\sF^{\rig}, \CE_\lambda) \rightarrow 0
\end{align*}
for all $i \in \BN$. Moreover, Assumption \ref{hypo1} and Corollary \ref{limitloc}, respectively, imply that the projective system of $K$-Fréchet spaces $\big(\bigoplus_{g \in G_0/P_I^m} H^{i}_{gY_I(\epsilon_m)}(\sF^{\rig}, \CE_\lambda)\big)_{m \in \BN}$ satisfies the topological Mittag-Leffler property for all $i\geq 0$ (cf. \cite[p. 626]{O2}). Therefore, by \cite[Remark 13.2.4]{EGAIII}, the ${\varprojlim_{m \in \BN}}^{(1)}$-term vanishes, i.e. 

$$ \Ext^i( \iota_*(\BZ_I), \CE_\lambda) \cong {\varprojlim_{m \in \BN}}\bigoplus_{g \in G_0/P_I^m} H^{i}_{gY_I(\epsilon_m)}(\sF^{\rig}, \CE_\lambda).$$
The statement of the proposition is then just rewriting the latter term. 
\end{proof}

\begin{proposition}
    We have a spectral sequence 
    \begin{equation*}\label{spectral2}
        {E}_1^{-p,q}=\bigoplus\limits_{\substack{I \subset \Delta \\ \lb \Delta \backslash I \rb = p }}\varprojlim_{m \in \BN} \Ind_{P_I^m}^{G_0} \big(H^q_{Y_I(\epsilon_m)}(\sF^{\rig},\CE_\lambda)\big)\Rightarrow H^{-p+q}(\sF^{\wa},\CE_\lambda)
    \end{equation*}
    where we use the abbreviation $Y_\Delta$ for $\sF$.
\end{proposition}
\begin{proof}
We follow the arguments used in the proof of \cite[Proposition 4.2]{O3}. First, we consider the second quadrant double complex $(\tilde{E}_1^{\bullet, \bullet},d^{\bullet, \bullet}, d'^{\bullet, \bullet})$ defined by 
$$ \tilde{E}_1^{p, q}= \begin{cases}  H^q(\sF^{\rig},\CE_\lambda) &\text{if } p=0 \\
    0 &\text{else, }   \end{cases}$$ 
with all differentials being trivial. Hence, it defines a spectral sequence converging to $H^*(\sF^{\rig},\CE_\lambda)$. Further, let $I \subset \Delta$ such that $\lb\Delta \backslash I\rb=1$ and $m \in \BN$. The inclusion $gY_I(\epsilon_m) \subset \sF^{\rig}$ induces a morphism (cf. \cite[Lemma 1.3]{vP}) 
$$ H^*_{gY_I(\epsilon_m)}(\sF^{\rig}, \CE_\lambda) \rightarrow H^*(\sF^{\rig}, \CE_\lambda)$$
for $g \in G/P^m_I$. Then, by the universal property of the direct sum we get a morphism 
$$ \bigoplus_{g \in G_0/P_I^m} H^{*}_{gY_I(\epsilon_m)^{\rig}}(\sF^{\rig}, \CE_\lambda) \rightarrow H^*(\sF^{\rig}, \CE_\lambda).$$
Thus, the functoriality of  ${\varprojlim_{m \in \BN}}$  yields 
$$  D_I^*:{\varprojlim_{m \in \BN}} \bigoplus_{g \in G_0/P_I^m} H^{*}_{gY_I(\epsilon_m)^{\rig}}(\sF^{\rig}, \CE_\lambda)\rightarrow  {\varprojlim_{m \in \BN}}  H^*(\sF^{\rig}, \CE_\lambda)=H^*(\sF^{\rig}, \CE_\lambda).$$ 
Then, we consider the morphism of double complexes (cf. (\ref{spectral}))
$$ f_1^{\bullet,\bullet}: \hat{E}_1^{{\bullet,\bullet}} \rightarrow \tilde{E}_1^{\bullet, \bullet}$$ 
given by 
$$ f_1^{p,q}=\begin{cases} \oplus_I D_I^q & \text{if }p=0 \\
   0 & \text{else.} \end{cases}$$ 
It induces the morphism of total complexes 
$$ \Tot(f_1^{\bullet,\bullet}): \Tot(E_1^{{\bullet,\bullet}}) \rightarrow \Tot(\tilde{E}_1^{\bullet, \bullet}) $$
where we denote the mapping cone of $\Tot(f_1^{\bullet,\bullet})$ by $\mathrm{Cone}(\Tot(f_1^{\bullet,\bullet})) ^\bullet$. By the definitions, the triangle for this mapping cone induces a long exact sequence which identifies with 
$$ \ldots \rightarrow H^q_{Y}(\sF^{\rig},\CE_\lambda) \rightarrow 
H^q(\sF^{\rig},\CE_\lambda) \rightarrow H^q(\sF^{\wa},\CE_\lambda)\rightarrow \ldots. $$ 
Hence, the cohomology of $\mathrm{Cone}(\Tot(f_1^{\bullet,\bullet})) ^\bullet$ coincides with $H^*(\sF^{\wa},\CE_\lambda)$. Furthermore, the total complex of the double complex ${E}_1^{\bullet,\bullet}$ in the statement is exactly $\mathrm{Cone}(\Tot(f_1^{\bullet,\bullet})) ^\bullet$ which finishes the proof. 
\end{proof}

Before stating and proving the main theorem we need the following lemma.

\begin{lemma}\label{help} Let $I \subsetneq \Delta$ and $w \in W^I \cap \Omega_I$. Then, $w \in \Omega_\emptyset$. 
\end{lemma}
\begin{proof}
    By Lemma \ref{n_a}, we know that $$\mu=\sum_{\alpha \in \Delta} n_\alpha \alpha^{\vee}$$ for $n_\alpha \in \BQ_{>0}$. 
    As $W$ acts by permutation on $\Phi$, we have  that $$w\mu=\sum_{\alpha \in \Delta} m_\alpha \alpha^{\vee}$$ with $m_\alpha \in \BQ$ for $w \in W^I \cap \Omega_I$.
    Then, by Lemma \ref{fweightandpairing} it is enough to show that $ \langle \check{\varpi}_{\alpha}, w\mu  \rangle_\der = m_\alpha>0$ for all $\alpha \in I$. By (\ref{linearcombi}), we have 
    $$ \alpha = \sum_{\beta \in \Delta} \langle \alpha, \beta^\vee{}\rangle_\der \check{\varpi}_\beta$$
    for $\alpha \in \Delta$. Moreover, from Lemma \ref{Kostant} we know that $w^{-1}\alpha \in \Phi^+ $ for all $\alpha \in I$. Since we assumed $\mu$ to lie in the positive Weyl chamber (cf. (\ref{positiveChamber})), we have 
    $$\langle \alpha, w\mu \rangle= \langle w^{-1}\alpha, \mu \rangle> 0$$ for $\alpha \in I$. By the very definition of $\langle \text{ , } \rangle_\der$ it follows that 
    $$ \langle \alpha, w\mu \rangle_\der > 0$$
    for all $\alpha \in I$. 
    Furthermore, for $\alpha, \beta \in \Delta$ and $\alpha \neq \beta$, we know by Lemma \ref{rootrel} that $\langle \alpha, \beta^\vee \rangle_\der \leq 0$. Recall that $w \in W^I \cap \Omega_I$ implies that $\langle \check{\varpi}_\beta  , w\mu \rangle_\der>0$ for $\beta \in \Delta \backslash I$ by Lemma \ref{fweightandpairing}. 
    Thus, for $\alpha \in I$ and $w \in W^I \cap \Omega_I$, we have that  
    \begin{equation}\label{b}
        b_\alpha:=\sum_{\beta \in I} \langle \alpha,  \beta^{\vee} \rangle_\der \langle \check{\varpi}_\beta, w\mu  \rangle_\der = \langle \alpha , w\mu \rangle_\der - \sum_{\beta \in \Delta \backslash I} \langle \alpha, \beta^{\vee} \rangle_\der \langle \check{\varpi}_\beta  , w\mu \rangle_\der > 0.
     \end{equation}
    We fix an ordering on $I=\{\alpha_1 > \alpha_2 > \ldots > \alpha_r \}$ and define 
    \begin{align*}
    &C \in \BQ^{\lb I \rb \times \lb I \rb } \text{ with } C_{ij}:=\langle \alpha_i, \alpha_{j}^\vee  \rangle_\der,  \\ 
    &x:=(m_{\alpha_i})_{i \in \{1, \ldots, r\}} \in \BQ^{\lb I \rb}, \\
     &b:=(b_{\alpha_i})_{i \in \{1, \ldots, r\}} \in \BQ^{\lb I \rb}.
    \end{align*}
    Then, 
    \begin{equation}\label{LGS}
    Cx=b. 
    \end{equation}
    After reordering the simple roots, if necessary, we can assume that $C$ has blocks $C_1,\ldots, C_t$ on the main diagonal and has zeroes everywhere else.
    Then, the $C_i$'s are the (transposed) Cartan matrices (cf. (\ref{cartanmatrix})) of the irreducible components of the Dynkin diagram of $\Phi_I$. Thus, $C^{-1}$ has blocks $C_i^{-1}$ on the main diagonal
    and has zeroes everywhere else. The entries of the $C_i^{-1}$ are, by Lemma \ref{inversecartan}, known to be positive rational. Then, (\ref{b}) and (\ref{LGS}) imply immediately that $m_\alpha>0$ for all $\alpha \in I$.
\end{proof}

    




\begin{theorem}\label{theorem1}
    Let $i_0:=\dim\sF-\lb\Delta\rb$. The homology of the (chain) complex 
    $$C_\bullet:  \bigoplus_{\substack{w \in \Omega_\emptyset \\ l(w)=\dim Y_\emptyset}}V^G_B(w) \leftarrow \ldots \leftarrow \bigoplus_{\substack{w \in \Omega_\emptyset \\ \ l(w)=1}} V_B^G(w) \leftarrow  V^G_B(\lambda) $$
    starting in degree $i_0$ coincides with $H^*(\sF^{\wa},\CE_\lambda)'$, i.e. $H_i(C_\bullet)=H^i(\sF^{\wa},\CE_\lambda)'$.
\end{theorem}

\begin{proof}

We consider the double complex $D_{\bullet, \bullet}$, similar to the one from \cite[p. 662]{OSch}, defined as  a second quadrant double chain complex, given by 
\begin{equation}\label{double}
D_{p,q}= \bigoplus_{{\substack{I \subset \Delta \\ \lb \Delta \backslash I \rb = -p}}} \bigoplus_{{\substack{w \in W^I \cap \Omega_I \\ l(w)=n-q}}} I_{P_I}^G(w) \, \Big(=\bigoplus_{{\substack{I \subset \Delta \\ \lb \Delta \backslash I \rb = -p}}} \bigoplus_{{\substack{w \in W^I \cap \Omega_I \\ l(w)=n-q}}}\CF^G_{P_I}\big(H^q_{C_I(w)}(\sF,\CE_\lambda)\big)\Big)
\end{equation}
(cf. (\ref{twisted}) for the objects). The vertical differentials are the ones coming from Lemma \ref{cohY_I}. The horizontal ones come from the transition maps $$H^q_{C_I(w)}(\sF,\CE_\lambda) \rightarrow H^q_{C_{I'}(w)}(\sF,\CE_\lambda)$$ for $I \subset I'$ and $w \in W^{I'}$ induced by the fact that $C_I(w) \subset C_{I'}(w)$ is a closed subset. They are the same as in Example \ref{genVM}. The commutativity is shown as in the proof of \cite[Theorem 4.2]{OSch}. \\
We are especially interested in the two spectral sequences converging towards the homology of the total complex $\Tot(D_{\bullet, \bullet})$ associated to $D_{\bullet, \bullet}$.  Namely,  
\begin{align*} {}^IE_{p,q}^0=D_{p,q} &\Rightarrow H_{p+q}(\Tot(D_{\bullet, \bullet})), \\
               {}^{II}E_{p,q}^0=D_{q,p} &\Rightarrow H_{p+q}(\Tot(D_{\bullet, \bullet})). 
\end{align*}
Then, by Lemma \ref{cohY_I} and Lemma \ref{Bi-Functor} in combination with the functoriality and the exactness of the functor $\CF_P^G$ (cf. \ref{exact}), 
we see that ${}^IE^1_{\bullet,\bullet}=(E_1^{\bullet,\bullet})'$ (cf. Proposition \ref{spectral2}). We know from Proposition \ref{extgroups} that the entries of $E_1^{\bullet,\bullet}$ are $K$-Fréchet spaces. Furthermore, the duality functor is exact on the category of $K$-Fréchet spaces (cf. \cite[Section I, Corollary 1.4]{BS}). Hence, $H_{p}(\Tot(D_{\bullet, \bullet}))\cong H^{p}(\sF^{\wa},\CE_\lambda)'.$ Due to Lemma \ref{help}, we have $${}^{II}E^0_{p,\bullet}=\bigoplus_{\substack{w \in \Omega_\emptyset \\ l(w)=n-p}} E^{0,w}_{p,\bullet}$$
with chain complexes
\begin{equation} \label{reso}
    E^{0,w}_{p,\bullet}:  I_{P_{I(w)}}^G(w) \rightarrow \bigoplus_{\substack{ I \subset I(w) \\ \lb I(w)\backslash I \rb =1 }} I_{P_{I}}^G(w)\rightarrow \ldots \rightarrow \bigoplus_{\substack{ I \subset I(w) \\ \lb I \rb =1 }} I_{P_{I}}^G(w) \rightarrow I_B^G(w)
\end{equation}
ending in degree $-\lb \Delta \rb$. From Corollary \ref{relativeresolution}, we know that these complexes are exact except at the very right position where the cokernel is $V_B^G(w)$. Thus, we get $${}^{II}E^1_{p,q}= \begin{cases} \bigoplus_{\substack{ w \in \Omega_\emptyset \\ l(w)=n-p }} V_{B}^G(w) &\text{ if } q=-\lb \Delta \rb, \\
0 &\text{ else.}   
\end{cases}$$
Therefore, ${}^{II}E^2={}^{II}E^\infty$ and we are done. 
\end{proof}
\begin{corollary}
    Let $i_0:=\dim\sF-\lb\Delta\rb$. Then, 
    $H^{i_0}(\sF^{\wa}, \CE_\lambda) \neq 0.$ 
\end{corollary}

\begin{proof}
   We know from \cite[Corollary 4.3]{OSch} that 
   $$v^G_B(\lambda)=\Ker\Big(V^G_B(\lambda) \rightarrow \bigoplus_{\substack{w \in W \\ \ l(w)=1}}V^G_B(w)\Big).$$ 
   But then it follows from the previous theorem that 
     $$v^G_B(\lambda)=\Ker\Big(V^G_B(\lambda) \rightarrow \bigoplus_{\substack{w \in W \\ \ l(w)=1}}V^G_B(w)\Big) \subset \Ker\Big(V^G_B(\lambda) \rightarrow \bigoplus_{\substack{w \in \Omega_\emptyset \\ \ l(w)=1}}V^G_B(w)\Big)=H^{i_0}(\sF^{\wa}, \CE_\lambda)'.$$ 
   Therefore, $H^{i_0}(\sF^{\wa}, \CE_\lambda)$ cannot be trivial. 
\end{proof}

\begin{lemma} \label{surjection} Let $w, w' \in \Omega_\emptyset$ with $w' \leq w$ and $l(w)=l(w')+1$. Then, the morphism 
    $$p_{w',w}:V^G_B(w')\rightarrow V^G_B(w)$$
 appearing in the differentials of $C_\bullet$ is surjective. 
 \end{lemma}
 \begin{proof} As seen in the proof of Theorem \ref{theorem1}, the morphism $p_{w',w}:V^G_B(w')\rightarrow V^G_B(w)$ is induced by a morphism $\varphi:I_{B}^G(w') \rightarrow I_{B}^G(w)$. This one in turn comes from a non-trivial morphism $$i_{w,w'}:M(w \cdot \lambda )=H^{n-l(w)}_{C(w)}(\sF,\CE_\lambda) \rightarrow H^{n-l(w')}_{C(w')}(\sF,\CE_\lambda)=M(w' \cdot \lambda )$$ (cf. Remark \ref{morphismBGG}). Thus, $i_{w,w'}$ is injective (cf. \cite[p. 46]{B}) and therefore $\varphi=\CF^G_B(i_{w,w'})$ is surjective (cf. Proposition \ref{exact}). Then, we have the following commutative diagram 
 $$\begin{CD}
 I_{B}^G(w')@>\CF_B^G(i_{w,w'})>>  I_{B}^G(w) \\ 
 @VV\pi V @VV \pi V \\
 V^G_B(w')@>p_{w',w}>>V^G_B(w)\\ 
 \end{CD}$$
 where $\pi$ denote the natural projection onto the quotient. Since all morphism except $p_{w',w}$ in the commutative diagram are surjective, it follows that $p_{w',w}$ is also surjective. 
 \end{proof}

In the following examples, we will compute the composition factors of the homology groups of the complex $C_\bullet$ of Theorem \ref{theorem1} for $\bG=\SL_4$ and some $\mu \in X_*(\bT)$. The strategy is first to 
compute all composition factors with multiplicities of the objects in $C_\bullet$  with the help of Theorem \ref{multiplicities}. This is done with a small program in SAGE (cf. Appendix \ref{AppendixCode}). Then, we can deduce the composition factors of the homology groups by knowing by the previous lemma that the morphism  $p_{w',w}:V^G_B(w')\rightarrow V^G_B(w)$ in the complex $C_\bullet$ is surjective for $w',w \in W$ with  $w' \leq w$ and how composition factors behave under short exact sequences. 
\begin{definition}
    Let $D$ be a composition factor of $V^G_B(\lambda)$ and $n_w:=\big[V^G_B(w):D\big]$ the multiplicity of $D$ in $V^G_B(w)$ for $w \in W$. Then, we define the \textsl{distribution type} of $D$ in the complex $C_\bullet$ by 
    $$\big(n_e,\{n_w\}_{w \in \Omega_\emptyset, \, l(w)=1},\ldots,\{n_w\}_{w \in \Omega_\emptyset, \, l(w)=\dim Y_\emptyset}\big) \in \BN_0^{\lb \Omega_\emptyset\rb}.$$
\end{definition}
\begin{remark}
    The distribution type depends on an ordering on $\Omega_\emptyset$. We will implicitly give such an ordering in each example and hope that causes no confusion with the notation. 
\end{remark}
\begin{example}\label{exampleCohomology}
        Let $\bG=\GL_4$, $\Delta=\{\alpha_1, \alpha_2, \alpha_3 \}$, $S=\{s_1,s_2,s_3\} \subset W$ with $s_i$ corresponding to $\alpha_i$, and $s_1$ commutes with $s_3$. We set $\bP_i=\bP_{\{\alpha_i\}}$ and $\bP_{i,j}=\bP_{\{\alpha_i,\alpha_j\}}$.
        Furthermore let $\mu=(x_1,x_2,x_3,x_4) \in X_*(\bT)\cong \BZ^4$ with $x_1>x_2>x_3>x_4$.
        \begin{enumerate}[label=\alph*),leftmargin=*]
            \item \label{exampleCohomologyA} $\mu=(x_1,x_2,x_3,x_4)$ with $\sum x_i=0$ and $x_3>0$. 
            Then $$\Omega_\emptyset=\{e, s_1, s_2, s_1s_2,s_2s_1, s_1s_2s_1\}$$
            and $$C_\bullet: V_B^G(\lambda) \xrightarrow{f} \bigoplus_{\substack{w \in \Omega_\emptyset \\ \ l(w)=1}} V_B^G(w) \xrightarrow{g}  \bigoplus_{\substack{w \in \Omega_\emptyset \\ \ l(w)=2}} V_B^G(w)  \xrightarrow{h}  V^G_B(s_1s_2s_1).$$
            The appearing distribution types of $C_\bullet$ are 
            \begin{align*}
                &\big(\{2\},\{2,1\},\{1,1\},\{1\}\big),\, \big(\{2\},\{1,2\},\{1,1\},\{1\}\big), \\
                &\big(\{1\},\{1,1\},\{1,1\},\{1\}\big),\, \big(\{1\},\{1,1\},\{1,0\},\{0\}\big), \\
                &\big(\{1\},\{1,1\},\{0,1\},\{0\}\big),\, \big(\{1\},\{1,0\},\{0,0\},\{0\}\big),\\
                & \big(\{1\},\{0,1\},\{0,0\},\{0\}\big),\,  \big(\{1\},\{0,0\},\{0,0\},\{0\}\big).\\ 
            \end{align*}
            As an example for the computations, we consider the distribution type $$\big(\{2\},\{2,1\},\{1,1\},\{1\}\big)$$ and denote a corresponding factor by $D$. 
            Then, $$[\Ker(f):D] \leq [\Ker\big(V^G_B(\lambda)\rightarrow V^G_B(s_1)\big):D]=0.$$ This implies that 
            $[\Im(f):D]=2$. Moreover, $[\Im(h):D]=1$ since $$V^G_B(s_1s_2) \rightarrow V^G_B(s_1s_2s_1)$$ is surjective. Thus, 
            $[\Ker(h):D]=1$. As  the composition $$\bigoplus_{\substack{w \in \Omega_\emptyset \\ \ l(w)=1}} V_B^G(w) \stackrel{g}{\rightarrow} \bigoplus_{\substack{w \in \Omega_\emptyset \\ \ l(w)=2}} V_B^G(w) \stackrel{\pi_1}{\rightarrow} V^G_B(s_1s_2)$$
            is surjective, we have the chain of inequalities 
            $$ 1\leq[\Im(g):D]\leq [\Ker(h):D]=1.$$ 
            Therefore, $[\Im(g):D]=1$ and $[\Ker(g):D]=2$. Finally, we see that 
            $$ [H_i(C_\bullet):D]=0$$ for all $i$. The same arguments applied to all distribution types show that 
            $H_i(C_\bullet)=0$ for $i \neq \dim(\sF)-\lb \Delta \rb=3$ and that $H_3(C_\bullet)=\Ker(f)$ has composition factors precisely
            \begin{equation*}v^G_B(\lambda),\, \CF_{P_{1,2}}^G\Big(L(s_3\cdot \lambda),v^{P_{1,2}}_{B}\Big), \CF_{P_{1,3}}^G\Big(L(s_2s_3\cdot \lambda),v^{P_{1,3}}_{P_{3}}\Big),\,\CF_{P_{2,3}}^G\Big(L(s_1s_2s_3\cdot \lambda),1\Big)
            \end{equation*}
            each with multiplicity one.
            \item \label{exampleCohomologyB}$\mu=(x_1,x_2,x_3,x_4)$ with $\sum x_i=0$ and $x_3=0$.  Then, $$\Omega_\emptyset=\{e, s_1, s_2, s_2s_1 \}$$
            and $$C_\bullet: V_B^G(\lambda) \xrightarrow{f}  \bigoplus_{\substack{w \in \Omega_\emptyset \\ \ l(w)=1}} V_B^G(w) \xrightarrow{g}  V^G_B(s_2s_1). $$ 
            The appearing distribution types in $C_\bullet$ are 
            \begin{align*}
                &\big(\{2\},\{2,1\},\{1\}\big),\, \big(\{2\},\{1,2\},\{1\}\big), \\
                &\big(\{1\},\{1,1\},\{1\}\big),\, \big(\{1\},\{1,1\},\{0\}\big), \\
                & \big(\{1\},\{1,0\},\{0\}\big),\, \big(\{1\},\{0,1\},\{0\}\big),\\
                &  \big(\{1\},\{0,0\},\{0\}\big).
            \end{align*}
            With the same arguments as above, we get that 
            $H_i(C_\bullet)=0$ for $i \neq 2,3$. Furthermore, $H_3(C_\bullet)$ has composition factors precisely
            \begin{align*}  &v^G_B(\lambda),\, \CF_{P_{1,2}}^G\Big(L(s_3\cdot \lambda),v^{P_{1,2}}_{B}\Big), \\ &\CF_{P_{1,3}}^G\Big(L(s_2s_3\cdot \lambda),v^{P_{1,3}}_{P_{3}}\Big),\,\CF_{P_{2,3}}^G\Big(L(s_1s_2s_3\cdot \lambda),1\Big)
            \end{align*}
            each with multiplicity one. Moreover, $H_2(C_\bullet)$ has composition factors precisely
            \begin{align*}  &\CF_{P_{2,3}}^G\Big(L(s_1s_2\cdot \lambda),v^{P_{2,3}}_{B}\Big), \CF_{P_{2,3}}^G\Big(L(s_1s_2s_3\cdot \lambda),v^{P_{2,3}}_{B}\Big), \\ 
                            &\CF_{P_{2}}^G\Big(L(s_3s_1s_2\cdot \lambda),v^{P_{2}}_{B}\Big),\,\CF_{P_{2}}^G\Big(L(s_1s_2s_3s_2\cdot \lambda),v^{P_{2}}_{B}\Big), \\
                            &\CF_{P_{1,3}}^G\Big(L(s_2s_3s_1s_2\cdot \lambda),1\Big),\,\CF_{P_{1,3}}^G\Big(L(s_2s_3s_1s_2\cdot \lambda),v^{P_{1,3}}_{P_{3}}\Big), \\
                            &\CF_{P_{3}}^G\Big(L(s_1s_2s_3s_1s_2\cdot \lambda),1\Big)
            \end{align*}
            each with multiplicity one as well. 
        \item \label{exampleCohomologyC}$\mu=(x_1,x_2,x_3,x_4)$ with $\sum x_i=0$, $x_2>0,x_3<0, x_1+x_4>0, x_2+x_3<0$.  Then,$$\Omega_\emptyset=\{e, s_1, s_2, s_3, s_1s_3, s_2s_3  \}$$
        and $$C_\bullet: V_B^G(\lambda) \xrightarrow{f} \bigoplus_{\substack{w \in \Omega_\emptyset \\ \ l(w)=1}} V_B^G(w) \xrightarrow{g}  V^G_B(s_1s_3) \oplus V^G_B(s_2s_3). $$ 
        The appearing distribution types in $C_\bullet$  are 
        \begin{align*}
            &\big(\{2\},\{2,1,2\},\{2,1\}\big),\, \big(\{2\},\{1,2,1\},\{1,1\}\big),\, \big(\{1\},\{1,1,1\},\{1,1\}\big), \\
            &\big(\{1\},\{1,1,1\},\{1,0\}\big),\, \big(\{1\},\{1,0,1\},\{1,0\}\big),\, \big(\{1\},\{0,1,1\},\{0,1\}\big), \\
            &\big(\{1\},\{1,1,0\},\{0,0\}\big),\, \big(\{1\},\{0,1,1\},\{0,0\}\big),\, \big(\{1\},\{1,0,0\},\{0,0\}\big), \\
            &\big(\{1\},\{0,1,0\},\{0,0\}\big),\, \big(\{1\},\{0,0,1\},\{0,0\}\big),\, \big(\{1\},\{0,0,0\},\{0,0\}\big).
        \end{align*}
        First, we notice that $g$ is surjective since $V^G_B(s_1)$ and $V^G_B(s_2)$ map onto a single but distinct direct summand. Then, we can apply the same arguments as before. We compute that  
        $H_i(C_\bullet)=0$ for $i \neq 2,3$. Furthermore, $H_3(C_\bullet)=v^G_B(\lambda)$. 
     Moreover, $H_2(C_\bullet)$ has composition factors precisely
        \begin{align*}  &\CF_{P_{2,3}}^G\Big(L(s_1s_2\cdot \lambda),v^{P_{2,3}}_{B}\Big),\, \CF_{P_{1,3}}^G\Big(L(s_2s_1\cdot \lambda),v^{P_{1,3}}_{B}\Big),\\
                        &\CF_{P_{1,2}}^G\Big(L(s_3s_2\cdot \lambda),v^{P_{1,2}}_{B}\Big),\, \CF_{P_{3}}^G\Big(L(s_1s_2s_1\cdot \lambda),v^{P_{3}}_{B}\Big),\\
                        &\CF_{P_{1,2}}^G\Big(L(s_3s_2s_1\cdot \lambda),v^{P_{1,2}}_{B}\Big), \, \CF_{P_{1,2}}^G\Big(L(s_3s_2s_1\cdot \lambda),v^{P_{1,2}}_{P_2}\Big), \\
                        &\CF_{P_{2}}^G\Big(L(s_3s_1s_2\cdot \lambda),v^{P_{2}}_{B}\Big), \,\CF_{P_{1,3}}^G\Big(L(s_2s_3s_1s_2\cdot \lambda),1\Big),\\
                        &\CF_{P_{1,3}}^G\Big(L(s_2s_3s_1s_2\cdot \lambda),v^{P_{1,3}}_{P_{1}}\Big), \, \CF_{P_{2}}^G\Big(L(s_3s_1s_2s_1\cdot \lambda),1\Big),\\
                        &\CF_{P_{2}}^G\Big(L(s_3s_1s_2s_1\cdot \lambda),v^{P_{2}}_{B}\Big),\, \CF_{P_{1}}^G\Big(L(s_2s_3s_1s_2s_1\cdot \lambda),1\Big)
        \end{align*}
        each with multiplicity one. 
        \end{enumerate}
\end{example}
\appendix 
\section{Code for composition factors}\label{s:Appendix} \label{AppendixCode}
	 Here, we present the code used for the computation of the Jordan-Hölder factors  of $V^G_B(w)$ with multiplicities from Example \ref{exampleCohomology}. The chosen language is SAGE: \\

	 compositionfactors.sage: 
\begin{lstlisting}[numbers=left, numberstyle=\tiny, xleftmargin= 25pt]
R.<q>=LaurentPolynomialRing(QQ)                                           
KL=KazhdanLusztigPolynomial(W,q)                                          
	 
def supp(W,w):
    supp=set([])
    ref=W.bruhat_interval(1,w)
    for v in ref:
        if v.length()==1:
           supp.add(v)
        return Set(supp)
	 
def I(W,w):
    I=set({})
    for s in W.simple_reflections():
        if (s*w).length()>w.length():
           I.add(s)
    return Set(I) 
	 
def multiplicity(W,w,v,J): 
    x=W.long_element()
    H=I(W,w)
    M=H.intersection(J)
    c=M.cardinality() 
    mult=0 
    ref=W.bruhat_interval(W.one(),v) 
    ref1=[]
    for t in ref:
        ref1.append(t*w.inverse())
    for t in ref1:
        if supp(W,t)==M:
           mult=mult+pow(-1,t.length()+c)*KL.P(x*t*w*x,x*v*x)(1)
    return mult 
	 
def multiplicitytot(W,w):
    res=[]
    c=0
    L=W.bruhat_interval(W.one(),W.long_element())
    for v in L: 
        H=I(W,v)
        S=Subsets(H)
        for J in S:
            m=multiplicity(W,w,v,J)
            if m != 0:
               c=c+1
               h=[]
               h.append(v)
               h.append(H) 
               h.append(J)
               h.append(m)
               res.append(h)
    return res, c 
	\end{lstlisting}
	Then, we applied this part to the relevant Weyl group elements. 
	 \begin{lstlisting}[numbers=left, numberstyle=\tiny, xleftmargin= 25pt]
sage: W = WeylGroup("A3",prefix="s")                                            
sage: [s1,s2,s3]=W.simple_reflections()   
sage: load("compositionfactors.sage")                                      
sage: multiplicitytot(W,W.one())                                                                             
  ([[s1*s2*s3*s1*s2*s1, {}, {}, 1],
    [s2*s3*s1*s2*s1, {s1}, {}, 1],
    [s2*s3*s1*s2*s1, {s1}, {s1}, 1],
    [s1*s2*s3*s2*s1, {s2}, {}, 2],
    [s1*s2*s3*s2*s1, {s2}, {s2}, 1], 
    [s1*s2*s3*s1*s2, {s3}, {}, 1],
    [s1*s2*s3*s1*s2, {s3}, {s3}, 1],
    [s3*s1*s2*s1, {s2}, {}, 1],
    [s3*s1*s2*s1, {s2}, {s2}, 1],
    [s2*s3*s2*s1, {s1}, {}, 1],
    [s2*s3*s2*s1, {s1}, {s1}, 1],
    [s2*s3*s1*s2, {s1, s3}, {}, 2],
    [s2*s3*s1*s2, {s1, s3}, {s1}, 1],
    [s2*s3*s1*s2, {s1, s3}, {s3}, 1],
    [s2*s3*s1*s2, {s1, s3}, {s1, s3}, 1],
    [s1*s2*s3*s1, {s3}, {}, 1],
    [s1*s2*s3*s1, {s3}, {s3}, 1],
    [s1*s2*s3*s2, {s2}, {}, 1],
    [s1*s2*s3*s2, {s2}, {s2}, 1],
    [s1*s2*s1, {s3}, {}, 1],
    [s3*s2*s1, {s1, s2}, {}, 1],
    [s3*s2*s1, {s1, s2}, {s1}, 1],
    [s3*s2*s1, {s1, s2}, {s2}, 1],
    [s3*s2*s1, {s1, s2}, {s1, s2}, 1],
    [s3*s1*s2, {s2}, {}, 1],
    [s3*s1*s2, {s2}, {s2}, 1],
    [s2*s3*s1, {s1, s3}, {}, 1],
    [s2*s3*s1, {s1, s3}, {s1}, 1],
    [s2*s3*s1, {s1, s3}, {s3}, 1],
    [s2*s3*s1, {s1, s3}, {s1, s3}, 1],
    [s2*s3*s2, {s1}, {}, 1],
    [s1*s2*s3, {s3, s2}, {}, 1],
    [s1*s2*s3, {s3, s2}, {s3}, 1],
    [s1*s2*s3, {s3, s2}, {s2}, 1],
    [s1*s2*s3, {s3, s2}, {s3, s2}, 1],
    [s2*s1, {s1, s3}, {}, 1],
    [s2*s1, {s1, s3}, {s1}, 1],
    [s1*s2, {s3, s2}, {}, 1],
    [s1*s2, {s3, s2}, {s2}, 1],
    [s3*s1, {s2}, {}, 1],
    [s3*s2, {s1, s2}, {}, 1],
    [s3*s2, {s1, s2}, {s2}, 1],
    [s2*s3, {s1, s3}, {}, 1],
    [s2*s3, {s1, s3}, {s3}, 1],
    [s1, {s3, s2}, {}, 1],
    [s2, {s1, s3}, {}, 1],
    [s3, {s1, s2}, {}, 1],
    [1, {s1, s3, s2}, {}, 1]],
    48)
sage: multiplicitytot(W,s1)                                                                                  
  ([[s1*s2*s3*s1*s2*s1, {}, {}, 1],
    [s2*s3*s1*s2*s1, {s1}, {}, 1],
    [s2*s3*s1*s2*s1, {s1}, {s1}, 1],
    [s1*s2*s3*s2*s1, {s2}, {}, 2],
    [s1*s2*s3*s2*s1, {s2}, {s2}, 1],
    [s1*s2*s3*s1*s2, {s3}, {}, 1],
    [s1*s2*s3*s1*s2, {s3}, {s3}, 1],
    [s3*s1*s2*s1, {s2}, {}, 1],
    [s3*s1*s2*s1, {s2}, {s2}, 1],
    [s2*s3*s2*s1, {s1}, {}, 1],
    [s2*s3*s2*s1, {s1}, {s1}, 1],
    [s2*s3*s1*s2, {s1, s3}, {}, 1],
    [s2*s3*s1*s2, {s1, s3}, {s1}, 1],
    [s2*s3*s1*s2, {s1, s3}, {s3}, 1],
    [s2*s3*s1*s2, {s1, s3}, {s1, s3}, 1],
    [s1*s2*s3*s1, {s3}, {}, 1],
    [s1*s2*s3*s1, {s3}, {s3}, 1],
    [s1*s2*s3*s2, {s2}, {}, 1],
    [s1*s2*s1, {s3}, {}, 1],
    [s3*s2*s1, {s1, s2}, {}, 1],
    [s3*s2*s1, {s1, s2}, {s1}, 1],
    [s3*s2*s1, {s1, s2}, {s2}, 1],
    [s3*s2*s1, {s1, s2}, {s1, s2}, 1],
    [s3*s1*s2, {s2}, {}, 1],
    [s2*s3*s1, {s1, s3}, {}, 1],
    [s2*s3*s1, {s1, s3}, {s1}, 1],
    [s2*s3*s1, {s1, s3}, {s3}, 1],
	[s2*s3*s1, {s1, s3}, {s1, s3}, 1],
    [s1*s2*s3, {s3, s2}, {}, 1],
    [s1*s2*s3, {s3, s2}, {s3}, 1],
    [s2*s1, {s1, s3}, {}, 1],
    [s2*s1, {s1, s3}, {s1}, 1],
    [s1*s2, {s3, s2}, {}, 1],
    [s3*s1, {s2}, {}, 1],
    [s1, {s3, s2}, {}, 1]],
    35)
sage: multiplicitytot(W,s2)                                                                                  
  ([[s1*s2*s3*s1*s2*s1, {}, {}, 1],
    [s2*s3*s1*s2*s1, {s1}, {}, 1],
    [s2*s3*s1*s2*s1, {s1}, {s1}, 1],
    [s1*s2*s3*s2*s1, {s2}, {}, 1],
    [s1*s2*s3*s2*s1, {s2}, {s2}, 1],
    [s1*s2*s3*s1*s2, {s3}, {}, 1],
    [s1*s2*s3*s1*s2, {s3}, {s3}, 1],
    [s3*s1*s2*s1, {s2}, {}, 1],
    [s3*s1*s2*s1, {s2}, {s2}, 1],
    [s2*s3*s2*s1, {s1}, {}, 1],
    [s2*s3*s1*s2, {s1, s3}, {}, 2],
    [s2*s3*s1*s2, {s1, s3}, {s1}, 1],
    [s2*s3*s1*s2, {s1, s3}, {s3}, 1],
    [s2*s3*s1*s2, {s1, s3}, {s1, s3}, 1],
    [s1*s2*s3*s1, {s3}, {}, 1],
    [s1*s2*s3*s2, {s2}, {}, 1],
    [s1*s2*s3*s2, {s2}, {s2}, 1],
    [s1*s2*s1, {s3}, {}, 1],
    [s3*s2*s1, {s1, s2}, {}, 1],
    [s3*s2*s1, {s1, s2}, {s2}, 1],
    [s3*s1*s2, {s2}, {}, 1],
    [s3*s1*s2, {s2}, {s2}, 1],
    [s2*s3*s1, {s1, s3}, {}, 1],
    [s2*s3*s2, {s1}, {}, 1],
    [s1*s2*s3, {s3, s2}, {}, 1],
    [s1*s2*s3, {s3, s2}, {s2}, 1],
    [s2*s1, {s1, s3}, {}, 1],
    [s1*s2, {s3, s2}, {}, 1],
    [s1*s2, {s3, s2}, {s2}, 1],
    [s3*s2, {s1, s2}, {}, 1],
    [s3*s2, {s1, s2}, {s2}, 1],
    [s2*s3, {s1, s3}, {}, 1],
    [s2, {s1, s3}, {}, 1]],
    33)
sage: multiplicitytot(W,s1*s2)                                                      
  ([[s1*s2*s3*s1*s2*s1, {}, {}, 1],
    [s2*s3*s1*s2*s1, {s1}, {}, 1],
    [s2*s3*s1*s2*s1, {s1}, {s1}, 1],
    [s1*s2*s3*s2*s1, {s2}, {}, 1],
    [s1*s2*s3*s2*s1, {s2}, {s2}, 1],
    [s1*s2*s3*s1*s2, {s3}, {}, 1],
    [s1*s2*s3*s1*s2, {s3}, {s3}, 1],
    [s3*s1*s2*s1, {s2}, {}, 1],
    [s3*s1*s2*s1, {s2}, {s2}, 1],
    [s2*s3*s1*s2, {s1, s3}, {}, 1],
    [s2*s3*s1*s2, {s1, s3}, {s1}, 1],
    [s2*s3*s1*s2, {s1, s3}, {s3}, 1],
    [s2*s3*s1*s2, {s1, s3}, {s1, s3}, 1],
    [s1*s2*s3*s1, {s3}, {}, 1],
    [s1*s2*s3*s2, {s2}, {}, 1],
    [s1*s2*s1, {s3}, {}, 1],
    [s3*s1*s2, {s2}, {}, 1],
    [s1*s2*s3, {s3, s2}, {}, 1],
    [s1*s2, {s3, s2}, {}, 1]],
    19)
sage: multiplicitytot(W,s2*s1)                                                                               
  ([[s1*s2*s3*s1*s2*s1, {}, {}, 1],
    [s2*s3*s1*s2*s1, {s1}, {}, 1],
    [s2*s3*s1*s2*s1, {s1}, {s1}, 1],
    [s1*s2*s3*s2*s1, {s2}, {}, 1],
    [s1*s2*s3*s2*s1, {s2}, {s2}, 1],
    [s1*s2*s3*s1*s2, {s3}, {}, 1],
    [s3*s1*s2*s1, {s2}, {}, 1],
    [s3*s1*s2*s1, {s2}, {s2}, 1],
    [s2*s3*s2*s1, {s1}, {}, 1],
    [s2*s3*s1*s2, {s1, s3}, {}, 1],
    [s2*s3*s1*s2, {s1, s3}, {s1}, 1],
    [s1*s2*s3*s1, {s3}, {}, 1],
    [s1*s2*s1, {s3}, {}, 1],
    [s3*s2*s1, {s1, s2}, {}, 1],
    [s3*s2*s1, {s1, s2}, {s2}, 1],
    [s2*s3*s1, {s1, s3}, {}, 1],
    [s2*s1, {s1, s3}, {}, 1]],
    17)
sage: multiplicitytot(W,s1*s2*s1)                                                                            
  ([[s1*s2*s3*s1*s2*s1, {}, {}, 1],
    [s2*s3*s1*s2*s1, {s1}, {}, 1],
    [s2*s3*s1*s2*s1, {s1}, {s1}, 1],
    [s1*s2*s3*s2*s1, {s2}, {}, 1],
    [s1*s2*s3*s2*s1, {s2}, {s2}, 1],
    [s1*s2*s3*s1*s2, {s3}, {}, 1],
    [s3*s1*s2*s1, {s2}, {}, 1],
    [s3*s1*s2*s1, {s2}, {s2}, 1],
    [s2*s3*s1*s2, {s1, s3}, {}, 1],
    [s2*s3*s1*s2, {s1, s3}, {s1}, 1],
    [s1*s2*s3*s1, {s3}, {}, 1],
    [s1*s2*s1, {s3}, {}, 1]],
    12)
sage: multiplicitytot(W,s3)                                                                                  
  ([[s1*s2*s3*s1*s2*s1, {}, {}, 1],
    [s2*s3*s1*s2*s1, {s1}, {}, 1],
    [s2*s3*s1*s2*s1, {s1}, {s1}, 1],
    [s1*s2*s3*s2*s1, {s2}, {}, 2],
    [s1*s2*s3*s2*s1, {s2}, {s2}, 1],
    [s1*s2*s3*s1*s2, {s3}, {}, 1],
    [s1*s2*s3*s1*s2, {s3}, {s3}, 1],
    [s3*s1*s2*s1, {s2}, {}, 1],
    [s2*s3*s2*s1, {s1}, {}, 1],
    [s2*s3*s2*s1, {s1}, {s1}, 1],
    [s2*s3*s1*s2, {s1, s3}, {}, 1],
    [s2*s3*s1*s2, {s1, s3}, {s1}, 1],
    [s2*s3*s1*s2, {s1, s3}, {s3}, 1],
    [s2*s3*s1*s2, {s1, s3}, {s1, s3}, 1],
    [s1*s2*s3*s1, {s3}, {}, 1],
    [s1*s2*s3*s1, {s3}, {s3}, 1],
    [s1*s2*s3*s2, {s2}, {}, 1],
    [s1*s2*s3*s2, {s2}, {s2}, 1],
    [s3*s2*s1, {s1, s2}, {}, 1],
    [s3*s2*s1, {s1, s2}, {s1}, 1],
    [s3*s1*s2, {s2}, {}, 1],
    [s2*s3*s1, {s1, s3}, {}, 1],
    [s2*s3*s1, {s1, s3}, {s1}, 1],
    [s2*s3*s1, {s1, s3}, {s3}, 1],
    [s2*s3*s1, {s1, s3}, {s1, s3}, 1],
    [s2*s3*s2, {s1}, {}, 1],
    [s1*s2*s3, {s3, s2}, {}, 1],
    [s1*s2*s3, {s3, s2}, {s3}, 1],
    [s1*s2*s3, {s3, s2}, {s2}, 1],
    [s1*s2*s3, {s3, s2}, {s3, s2}, 1],
    [s3*s1, {s2}, {}, 1],
    [s3*s2, {s1, s2}, {}, 1],
    [s2*s3, {s1, s3}, {}, 1],
    [s2*s3, {s1, s3}, {s3}, 1],
    [s3, {s1, s2}, {}, 1]],
    35)
sage: multiplicitytot(W,s1*s3)                                                                               
  ([[s1*s2*s3*s1*s2*s1, {}, {}, 1],
    [s2*s3*s1*s2*s1, {s1}, {}, 1],
    [s2*s3*s1*s2*s1, {s1}, {s1}, 1],
    [s1*s2*s3*s2*s1, {s2}, {}, 2],
    [s1*s2*s3*s2*s1, {s2}, {s2}, 1],
    [s1*s2*s3*s1*s2, {s3}, {}, 1],
    [s1*s2*s3*s1*s2, {s3}, {s3}, 1],
    [s3*s1*s2*s1, {s2}, {}, 1],
    [s2*s3*s2*s1, {s1}, {}, 1],
    [s2*s3*s2*s1, {s1}, {s1}, 1],
    [s2*s3*s1*s2, {s1, s3}, {}, 1],
    [s2*s3*s1*s2, {s1, s3}, {s1}, 1],
    [s2*s3*s1*s2, {s1, s3}, {s3}, 1],
    [s2*s3*s1*s2, {s1, s3}, {s1, s3}, 1],
    [s1*s2*s3*s1, {s3}, {}, 1],
    [s1*s2*s3*s1, {s3}, {s3}, 1],
    [s1*s2*s3*s2, {s2}, {}, 1],
    [s3*s2*s1, {s1, s2}, {}, 1],
    [s3*s2*s1, {s1, s2}, {s1}, 1],
    [s3*s1*s2, {s2}, {}, 1],
    [s2*s3*s1, {s1, s3}, {}, 1],
    [s2*s3*s1, {s1, s3}, {s1}, 1],
    [s2*s3*s1, {s1, s3}, {s3}, 1],
    [s2*s3*s1, {s1, s3}, {s1, s3}, 1],
    [s1*s2*s3, {s3, s2}, {}, 1],
    [s1*s2*s3, {s3, s2}, {s3}, 1],
    [s3*s1, {s2}, {}, 1]],
    27)
sage: multiplicitytot(W,s2*s3)                                                                               
  ([[s1*s2*s3*s1*s2*s1, {}, {}, 1],
    [s2*s3*s1*s2*s1, {s1}, {}, 1],
    [s1*s2*s3*s2*s1, {s2}, {}, 1],
    [s1*s2*s3*s2*s1, {s2}, {s2}, 1],
    [s1*s2*s3*s1*s2, {s3}, {}, 1],
    [s1*s2*s3*s1*s2, {s3}, {s3}, 1],
    [s2*s3*s2*s1, {s1}, {}, 1],
    [s2*s3*s1*s2, {s1, s3}, {}, 1],
    [s2*s3*s1*s2, {s1, s3}, {s3}, 1],
    [s1*s2*s3*s1, {s3}, {}, 1],
    [s1*s2*s3*s2, {s2}, {}, 1],
    [s1*s2*s3*s2, {s2}, {s2}, 1],
    [s2*s3*s1, {s1, s3}, {}, 1],
    [s2*s3*s2, {s1}, {}, 1],
    [s1*s2*s3, {s3, s2}, {}, 1],
    [s1*s2*s3, {s3, s2}, {s2}, 1],
    [s2*s3, {s1, s3}, {}, 1]],
    17)
sage: multiplicitytot(W,s3*s2)                                                                               
  ([[s1*s2*s3*s1*s2*s1, {}, {}, 1],
    [s2*s3*s1*s2*s1, {s1}, {}, 1],
    [s2*s3*s1*s2*s1, {s1}, {s1}, 1],
    [s1*s2*s3*s2*s1, {s2}, {}, 1],
    [s1*s2*s3*s2*s1, {s2}, {s2}, 1],
    [s1*s2*s3*s1*s2, {s3}, {}, 1],
    [s1*s2*s3*s1*s2, {s3}, {s3}, 1],
    [s3*s1*s2*s1, {s2}, {}, 1],
    [s2*s3*s2*s1, {s1}, {}, 1],
    [s2*s3*s1*s2, {s1, s3}, {}, 1],
    [s2*s3*s1*s2, {s1, s3}, {s1}, 1],
    [s2*s3*s1*s2, {s1, s3}, {s3}, 1],
    [s2*s3*s1*s2, {s1, s3}, {s1, s3}, 1],
    [s1*s2*s3*s2, {s2}, {}, 1],
    [s1*s2*s3*s2, {s2}, {s2}, 1],
    [s3*s2*s1, {s1, s2}, {}, 1],
    [s3*s1*s2, {s2}, {}, 1],
    [s2*s3*s2, {s1}, {}, 1],
    [s3*s2, {s1, s2}, {}, 1]],
    19)
	 
	\end{lstlisting}

\printbibliography[title={References}]
\end{document}